\documentclass[11pt]{amsart}
\usepackage[margin=1.25in]{geometry}

\usepackage{mathrsfs}
\usepackage{amsfonts}
\usepackage{latexsym,epsfig}
\usepackage{amsmath, amsthm, amscd, amssymb, centernot}
\usepackage[pdftex]{color}
\usepackage{comment}
\usepackage{marginnote}
\usepackage[colorlinks,citecolor=cyan,linkcolor=violet]{hyperref}
\usepackage[nameinlink]{cleveref}
\usepackage{enumitem}
\usepackage[normalem]{ulem}

\usepackage{tikz}
\usepackage{tikz-cd}
\usetikzlibrary{matrix,decorations.pathmorphing,arrows}
\tikzset{commutative diagrams/.cd,arrow style=tikz,diagrams={>=stealth'}}
\usepackage{todonotes}

\newcommand{\RR}[0]{\mathbb{R}}
\newcommand{\HH}[0]{\mathbb{H}}

\newcommand{\ZZ}[0]{\mathbb{Z}}

\newcommand{\R}{\mathbb{R}}

\renewcommand{\l}{\lambda}

\newcommand{\wt}{\widetilde}
\newcommand{\mc}{\mathcal}

\newcommand{\ol}{\overline}

\DeclareMathOperator{\fr}{fr}

\newcommand{\orb}{\mathcal{O}}

\newcommand\tsim{\kern-.4em\sim}

\newcommand\ssm{\smallsetminus}

\newcommand{\tilt}{\mathrm{tilt}}
\newcommand{\spn}{\mathrm{span}}

\renewcommand{\phi}{\varphi}
\renewcommand{\epsilon}{\varepsilon}

\newcommand{\defn}[1]{\textbf{\emph{#1}}}

\usepackage{hyperref}

\usepackage{thmtools}

\numberwithin{equation}{section}

\newtheorem{theorem}[equation]{Theorem}

\newtheorem{thm}{Theorem}

\newtheorem{lemma}[equation]{Lemma}
\newtheorem{proposition}[equation]{Proposition}
\newtheorem{corollary}[equation]{Corollary}

\newtheorem{claim}[equation]{Claim}

\theoremstyle{definition}

\newtheorem{example}[equation]{Example}

\declaretheorem[style=definition,qed=$\lozenge$,sibling=theorem]{remark}
\declaretheorem[style=definition,qed=$\lozenge$,sibling=theorem]{definition}

\begin{document}
\title{Universal circles for Anosov foliations}

\author[E. Buckminster]{Ellis Buckminster}
\address{Department of Mathematics\\
	University of Pennsylvania}
\email{\href{mailto:ellis17@sas.upenn.edu}{ellis17@sas.upenn.edu}}

\author[S.J. Taylor]{Samuel J. Taylor}
\address{Department of Mathematics\\ 
Temple University}
\email{\href{mailto:samuel.taylor@temple.edu}{samuel.taylor@temple.edu}}

\thanks{During the project, Buckminster was partially supported by a Simons Dissertation Fellowship, and
Taylor was partially supported by NSF grant DMS--2503113 and the Simons Foundation.}

\date{\today}

\begin{abstract}
Thurston introduced the notion of a \emph{universal circle} associated to a taut foliation of a $3$-manifold as a way of organizing the ideal circle boundaries of its leaves into a single circle action. Calegari--Dunfield proved that every taut foliation of an atoroidal $3$-manifold $M$ has a universal circle, but the uniqueness (or lack-thereof) of this structure remains rather mysterious. 

In this paper, we consider the foliations associated to an Anosov flow $\varphi$ on $M$, showing that several constructions of a universal circle in the literature are typically distinct. Moreover, the underlying action of the Calegari--Dunfield \emph{leftmost} universal circle is generally not even conjugate to the universal circle arising from the boundary of the flow space of $\varphi$. 
Our primary tool is a way to use the flow space of $\varphi$ to parameterize the circle bundle at infinity of $\varphi$'s invariant foliations.
\end{abstract}
\maketitle

\begin{figure}[h]
	\centering
	\includegraphics[width=.9\linewidth]{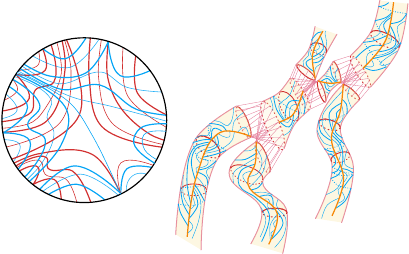}
	\caption{The orbit space of a non $\R$-covered Anosov flow, and a piece of the circle bundle at infinity of its unstable foliation.}\label{fig:coverart}
\end{figure}

\newpage

\setcounter{tocdepth}{1}
\tableofcontents

\section{Introduction}

Let $M$ be a closed, oriented, irreducible, and atoroidal (i.e. $\ZZ \oplus \ZZ \nleq \pi_1(M)$) $3$--manifold. A \defn{taut foliation} $\mc F$ is a cooriented foliation of $M$ by two-dimensional leaves with the property that each leaf is intersected by a closed curve that is positively transverse to $\mc F$. Such foliations have been central objects in the study of $3$--manifolds for decades \cite{Nov65, Gab83, thurston1997three, 2007CalegariBook} and have recently received a resurgence of interest due to their role in the $L$-space conjecture \cite{boyer2013spaces, ozsvath2004holomorphic, juhasz2015survey}. In this paper, we study the taut foliations that arise as the invariant foliations of an Anosov flow on $M$ (i.e. \defn{Anosov foliations}) and the lack of uniqueness of their \emph{universal circles}.
\smallskip

If $\mc F$ is a taut foliation of an atoroidal $3$--manifold $M$, then there is a metric on $M$ that restricts to a hyperbolic metric on each leaf of $\mc F$ \cite{Candel_uniformization}. Lifting to the universal cover $\wt M$, each leaf $\lambda$ of the lifted foliation $\wt {\mc F}$ admits a well-defined ideal boundary $\partial_\infty \lambda$ homeomorphic to a circle, and the fundamental group of $M$ acts on this collection of ideal circles. A \defn{universal circle} $\mc C$ for $\mc F$, originally introduced by Thurston, is a natural way of organizing this data into a \emph{single} circle action $\pi_1(M) \curvearrowright \mc C$. The action comes equipped with a family of monotone maps $\pi_\lambda \colon \mc C \to \partial_\infty \lambda$, one for each leaf $\lambda$ of $\wt {\mc F}$, that intertwines the universal circle action with the action on the disjoint union of all $\partial_\infty \lambda$. We refer the reader to \Cref{sec:ucs} for the complete definition, which was precisely formulated by Calegari--Dunfield \cite{CalDun_UC}. They additionally proved every taut foliation of an atoroidal manifold admits a universal circle.

Despite their role as a basic object in the study of flows and foliations, very little is know about the possible universal circles for a fixed foliation, or under what conditions a \emph{minimal} universal circle is unique. See, for example, \cite[Remark 6.24]{CalDun_UC} or Question 8.9 on Calegari's list \cite{calegari2002problems}. 
Indeed, even the Calegari--Dunfield construction itself has built in flexibility:
as one navigates the various ideal circle boundaries one can choose to take either the leftmost or rightmost turns, thereby leading to both a leftmost universal circle $\mc C^\ell$ and a rightmost universal circle $\mc C^r$; see \Cref{sec:CD} for the details.

More recently, the second author together with Landry and Minsky \cite{LMT_UC} showed that for any foliation $\mc F$ (almost) transverse to a pseudo-Anosov flow $\varphi$, the boundary $\partial \orb_\phi$ of $\varphi$'s flow space is naturally a universal circle for $\mc F$ (see \Cref{subsec:flowspace_uc_background}). This gives new constructions of universal circles, which are necessarily distinct for non orbit equivalent flows. 

However, there are situations where minimal universal circles are quite constrained. Calegari establishes a version of uniqueness for $\RR$--covered foliations \cite{Calegari2000Rcovered}
and foliations with one-sided branching \cite{calegari2003foliations}.
More recently, Huang proved that if $\mc F$ is a depth one foliation transverse to a pseudo-Anosov flow $\varphi$ without perfect fits, then the leftmost universal circle $\mc C^\ell$, the rightmost universal circle $\mc C^r$, and the universal circle $\partial \orb_\phi$ associated to the flow space $\orb_\phi$ are all isomorphic as universal circles \cite{Huang_UC}.

\medskip

We will consider the invariant foliations of an orientable Anosov flow and essentially come to the opposite conclusion: for such foliations all three universal circles are distinct. We next turn to more formally summarize our results.

\subsection*{Results}
Throughout the paper, $M$ is a closed, oriented, atoroidal $3$-manifold and $\varphi$ is a Anosov flow on $M$ with orientable stable/unstable foliations, which are denoted $W^{s/u}$. We remark that for an arbitrary Anosov flow on $M$, its invariant foliations become orientable after taking a double cover. Although our results hold for either invariant foliation, we will state everything in terms of the unstable foliation $W^u$.

It is easy to see that the foliation $W^u$ of $M$ is taut, and so the Calegari--Dunfield construction gives two (potentially isomorphic) universal circles; the leftmost universal circle $\mc C^\ell$ and the rightmost universal circle $\mc C^r$.

Following Fenley, the foliation can also be \emph{tilted} by an isotopy that makes it transverse to $\varphi$. We denote this tilted foliation $W^u_\tilt$ and refer to \Cref{sec:tilting_fol} for additional details. By \cite{LMT_UC}, the boundary $\partial \orb$ of $\varphi$'s flows space $\orb$ is naturally a universal circle for $W^u_\tilt$ and thus for $W^u$; see \Cref{subsec:flowspace_uc_background}.

In his survey of Anosov flows, Potrie suggest exploring the connection between these two notions of a universal circle for Anosov foliations \cite[Section 6]{potrie2025anosov}.
Our first main theorem is that these constructions generally determine different universal circles:

\begin{thm} \label{th:intro_1}
For a non $\RR$--covered Anosov flow, 
the minimal universal circles $\mc C^\ell$, $\mc C^r$, and $\partial \orb$ are all distinct, up to isomorphism.
\end{thm}

The case of \Cref{th:intro_1} for an $\RR$--covered $\varphi$ is better understood and is summarized in \Cref{sec:skew}. 

In \Cref{th:intro_1}, two universal circles are isomorphic if the underlying actions are conjugate by a homeomorphism that intertwines the monotone maps to the $\partial_\infty \lambda$; see \Cref{sec:ucs}.
In fact, one can also tilt the foliation $W^u$ in the opposite direction, and this too results in a distinct universal circle, even though the underlying circle action is again $\pi_1(M) \curvearrowright \partial \orb$.

\smallskip
Furthermore, we prove the following result that shows the Calegari--Dunfield and flow space boundary constructions give different circles in the strongest sense:

\begin{thm} \label{th:intro_2} 
The universal circle actions $\pi_1 M \curvearrowright \partial \orb$ and $\pi_1(M) \curvearrowright \mc C^\ell$ are not conjugate.
\end{thm}

\begin{figure}
	\centering
	\includegraphics[width=0.7\linewidth]{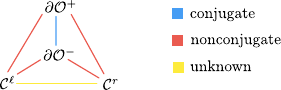}
	\caption{The relations between the four universal circles considered. The flow space boundary universal circles from tilting both ways are denoted $\partial\orb^\pm$. All four universal circles are nonisomorphic.}\label{fig:diagram}
\end{figure}

The results of \Cref{th:intro_1} and \Cref{th:intro_2} are summarized in \Cref{fig:diagram}.

\Cref{th:intro_2} is proven using a detailed analysis of Calegari--Dunfield construction applied to $W^u$. To explain, let $\mc L = \wt M / \wt W^u$ be the leaf space of the foliation $W^u$, i.e. the quotient space of the universal cover $\wt M$ obtained by identifying points contained in the same leaf of the lifted foliation $\wt W^u$.
To each leaf $\lambda \in \mc L$, there is a ideal circle $\partial_\infty \lambda$ and we let $E_\infty$ be the \emph{circle bundle at infinity} over $\mc L$ (also called the \emph{ideal tubulation} in \cite{potrie2025anosov}), which as a set is the union of the $\partial_\infty\l$. For details, see \Cref{sec:ucs}. 

At least morally, a universal circle is constructed by finding a sufficiently large $\pi_1(M)$--invariant collection of `noncrossing' sections $s\colon \mc L \to E_\infty$ of the circle bundle $E_\infty \to \mc L$. 
The noncrossing condition implies that these sections admit a natural circular order that is imposed by the circles $\partial_\infty \lambda$ for each $\lambda \in \mc L$. The associated universal circle $\mc C$ then arises from the completion of this circularly ordered set, and there is an induced action $\pi_1(M) \curvearrowright \mc C$ by orientation preserving homeomorphisms.
Hence, understanding the structure of $E_\infty$ is a key step toward understanding the possible universal circles for a foliation.

Since $W^u$ is the unstable foliation of the Anosov flow $\varphi$, the orbits in a given leaf are asymptotic in backwards time. This translates into the statement that for  each leaf $\lambda$ of $\wt W^u$, there is a unique point $\theta_\lambda \in \partial_\infty \lambda$ that is the backwards endpoint of each orbit in $\lambda$ of the lifted flow $\wt\varphi$. The assignment $\lambda \mapsto \theta_\lambda$ determines a (continuous) section $s_{\bf nm} \colon \mc L \to E_\infty$ that we call the \emph{nonmarker section} (\Cref{sec:markers}). The following result is our main tool to understand $E_\infty$. 

\begin{thm}\label{th:intro_3}
The assignment of each $p\in \lambda$ to the endpoint in $\partial_\infty \lambda$ of its forward $\wt\phi$--orbit induces an equivariant homeomorphism
\[
\Phi \colon \orb \to E_\infty \ssm s_{\bf nm}(\mc L).
\]
\end{thm}

As before, $\orb$ is the orbit space of $\phi$, defined as $\orb = \wt M / \wt \phi$ where $\wt \phi$ is the lifted flow. By work of Barbot \cite{barbot1995caracterisation} and Fenley \cite{fenley1994anosov}, $\orb$ is a plane bifoliated by $\orb^{u/s}$ which are the projections of the foliations $\wt W^{u/s}$ of $\wt M$. The importance of \Cref{th:intro_3} lies in the fact that $\Phi$ maps the foliations $\orb^{s/u}$ to two natural foliations of  $E_\infty \ssm s_{\bf nm}(\mc L)$. Essentially by definition, the unstable foliation $\orb^u$ is mapped to the foliation of $E_\infty \ssm s_{\bf nm}(\mc L)$ by open arcs $\partial_\infty \lambda \ssm \theta_\lambda$, for $\lambda \in \mc L$. More importantly, the stable foliation $\orb^s$ is mapped by $\Phi$ to the foliation by \emph{markers} (justifying the name given to $s_{\bf nm}$). Informally, markers are arcs in the $E_\infty$ determined by the directions where leaves of $\wt W^u$ are nonexpanding (see \Cref{sec:CD}) -- they are the central object in the Calegari--Dunfield construction of $\mc C^\ell$ and $C^r$. 

In this sense, \Cref{th:intro_3} together with a understanding of the foliations $\orb^{u/s}$ is the main tool to study the relation between $\partial \orb$ and leftmost and rightmost universal circles that is required to establish \Cref{th:intro_1} and \Cref{th:intro_2}. Indeed, Potrie suggests directly studying the action on the `more canonical' circle bundle at infinity $E_\infty$ \cite[Section 6]{potrie2025anosov}, and \Cref{th:intro_3} provides a direct approach for doing so in terms of the flow space $\orb$.

\subsection*{Outline of paper}
After reviewing the needed background in \Cref{sec:background}, \Cref{sec:cd} is concerned with universal circles and their constructions in the literature. \Cref{sec:tilting} gives a procedure to tilt the unstable foliations $W^u$ so that it becomes transverse to $\varphi$ and then studies the shadows of leaves in the flow space of $\varphi$. \Cref{sec:para} relates $E_\infty$ to the flow space $\orb$ and is where \Cref{th:intro_3} and its consequences are proved. These are the main technical tools of the paper. In \Cref{sec:skew}, we apply these results to the much easier (and essentially already known) case of skew Anosov flows. Finally, \Cref{th:intro_1} is proved in \Cref{sec:distinct_ucs} and \Cref{th:intro_2} is proved in \Cref{sec:nonconjugate}.

\subsection*{Acknowledgments}
We thank Thomas Barthelm\'e for help with using \cite[Corollary 5.23]{barthelme2023collapsed}, Katie Mann for helpful conversations about mapping special sections to the flow space, and Mladen Bestvina for asking whether the universal circles we consider have common covers. We also thank Junzhi Huang, Michael Landry, and Chi Cheuk Tsang for helpful comments on an earlier draft.
\section{Background}
\label{sec:background}

Here we collect some background needed throughout the paper.

\subsection{Circular orders and monotone maps}\label{sec:circle_background}

A \defn{circular order} on a set $X$ is a map
$\langle \cdot,\cdot,\cdot\rangle \colon X\times X\times X\to\{-1,0,1\}$
such that 
\begin{enumerate}
	\item $\langle a,b,c\rangle=\pm1$ if and only if $a,b,c$ are pairwise distinct, and
	\item for all $a,b,c,d\in X$, \[\langle b,c,d\rangle-\langle a,c,d\rangle +\langle a,b,d\rangle +\langle a,b,c\rangle=0.\]
\end{enumerate}

Intuitively, a nondegenerate ordered triple gets mapped to $1$ if it is ordered clockwise and to $-1$ if it is ordered counterclockwise. 

Whenever we have a circular order on an uncountable, second countable set $X$, we can complete it to a circle $\mc C$. We refer the reader to \cite{frankel2013quasigeodesic, bonatti2301action} for the precise details, but informally summarize this process as follows:
First, embed $X$ into the circle $S^1$ in a circular order preserving way. Take the closure of this set in $S^1$, obtaining a closed subset of $S^1$ whose complement is an at most countable union of open intervals. By collapsing the closures of these intervals, we obtain the circle $\mc C$. The map $i\colon X\to\mc C$ is an embedding so long as no point in $X$ is isolated on one or both sides. 

A \defn{monotone map} $\pi\colon\mc C_1\to\mc C_2$ between oriented circles is a degree one map such that the preimage of each point is connected. A maximal open interval in a point preimage is called a \defn{gap}, and the \defn{core} of $\pi$ is the complement (in $\mc C_1$) of its gaps.

\subsection{Anosov flows and their flow spaces}\label{sec:anosov_background}

Throughout the paper, we work in the category of \emph{topological} Anosov flows. Since our manifolds are atoroidal, such flows are transitive (see \cite[Proposition 2.6]{mosher1992dynamical} and more generally \cite[Corollary 1.6]{barthelme2024non}) and hence orbit equivalent to \emph{smooth} Anosov flows by work of Shannon \cite{Shannon20}. Here, an orbit equivalence is a homeomorphism sending oriented orbits to oriented orbits. Thus, our results also hold for smooth Anosov flows, but we will make use of the greater flexibility in the topological setting. We refer the reader to the first section of the recent book by Barthelm\'e and Mann \cite{barthelme2025pseudo} for the full story.

\medskip
To state the definition of a topological Anosov flow, following \cite{barthelme2025pseudo}, we first recall that if $W$ is a foliation of $M$ and $x \in M$, then $W_\epsilon(x)$ is the component of $W(x) \cap B_\epsilon(x)$ containing $x$. Here, $W(x)$ is the leaf through $x$ and $B_\epsilon(x)$ denotes the $\epsilon$--ball around $x$. Also, a standard reparameterization is an increasing homeomorphism $\sigma \colon \RR \to \RR$ that fixes $0$.

A \defn{(topological) Anosov flow} $\varphi$ on $M$ is a flow leaving invariant a pair of topologically transverse two-dimensional foliations $W^s$ and $W^u$ whose leaves intersect along orbits of $\varphi$ such that 
\begin{itemize}
\item For any $\epsilon >0$ sufficiently small, if $y \in W^s_\epsilon(x)$ (resp. $y \in W^u_\epsilon(x)$), there is a standard reparameterization $\sigma$ such that $d(\phi_t(x),\phi_{\sigma(t)}(y)) \to 0$ as $t \to \infty$ (resp. as $t\to -\infty)$. 
\item For any $\epsilon>0$ sufficiently small, there exits $\delta>0$ such that if $y\in W^s_\epsilon(x)$ (resp. $y \in W^u_\epsilon(x)$) is not in the same $\epsilon$--local orbit as $x$ then for any standard reparameterization $\sigma$, there exists $t<0$ (resp. $t>0$) with $d(\phi_t(x),\phi_{\sigma(t)}(y)) > \delta$.
\end{itemize}

The foliation $W^s$ is called the (weak) \defn{stable foliation} and $W^u$ is called the (weak) \defn{unstable foliation}.

\medskip
For any 2-dimensional foliation $\mc F$ on a 3-manifold $\wt M$, we define the \defn{leaf space} of $\mc F$ by $\mc L=\wt M/\wt{\mc F}$. In general, this is a simply connected, possibly non-Hausdorff 1-manifold. If $\mc L\cong\R$, we say $\mc F$ is \defn{$\R$-covered}, otherwise it is \defn{non $\R$-covered}. In the case of an Anosov flow, we denote the leaf space of $W^{s/u}$ by $\mc L^{s/u}$. Furthermore, we say the flow is \defn{$\R$-covered} if both $W^s$ and $W^u$ are, and \defn{non $\R$-covered} otherwise.

\subsubsection{The flow space and its boundary} An important tool in the study of Anosov flows is the \defn{flow space} (also called the \defn{orbit space}), originally studied by Barbot \cite{barbot1995caracterisation} and Fenley \cite{fenley1994anosov}. The flow space of an Anosov flow $\varphi$, denoted $\orb$, is defined as 
\[\orb=\wt M/\wt \varphi,\]
where $\wt M$ is the universal cover of $M$, and $\wt \varphi$ is a lift of $\varphi$ to $\wt M$. Let \[\Theta\colon\wt M\to\orb\] denote projection to the flow space. The flow space is topologically a plane and inherits an action by $\pi_1(M)$ coming from the action by deck transformations. The lifted weak stable and unstable (two-dimensional) foliations $\wt W^{s/u}$ descend to a pair of one-dimensional foliations \[\orb^{s/u}=\Theta(\wt W^{s/u}),\] called the \defn{stable and unstable foliations of $\orb$}. In figures, we will always draw $\orb^s$ as blue and $\orb^u$ as red. These foliations are preserved by the $\pi_1(M)$ action and can be quite complicated; see \Cref{fig:orbitspace}. 

Since $W^s$ and $W^u$ are orientable by assumption, the foliations $\orb^{s/u}$ admit a $\pi_1(M)$--invariant orientation. Using the induced orientation on $\orb$, we fix orientations on $\orb^{s/u}$ so that $\orb$ is oriented by charts taking leaves of $\orb^u$ to the oriented $x$--axis and leaves of $\orb^s$ to the oriented $y$--axis.

\begin{figure}[h]
	\centering
	\includegraphics[width=0.5\linewidth]{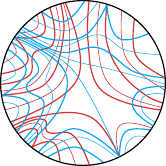}
	\caption{The orbit space of a non $\R$-covered Anosov flow.}\label{fig:orbitspace}
\end{figure}

\smallskip

Topologically, the flow space with its two foliations falls into one of three categories; see \cite{barbot1995caracterisation}, \cite{fenley1994anosov}, and \cite[Theorem 2.6.1]{barthelme2025pseudo}.
It is either
\begin{enumerate}
	\item \defn{trivial}, i.e. homeomorphic to $\R^2$ with the horizontal and vertical foliations,
	\item \defn{skew}, i.e. homeomorphic to the strip 
	\[\{(x,y)\in\R^2\:\mid\:x-1<y<x\}\] with the horizontal and vertical foliations, or
	\item \defn{branching}, meaning the foliations $\orb^s$ and $\orb^u$ are both non $\R$-covered.
\end{enumerate}

Since $\mc L^{s/u}\cong \orb/\orb^{s/u}$ by a canonical homeomorphism, we don't distinguish between the leaf spaces of $W^s$ (respectively $W^u$) and $\orb^s$ (respectively $\orb^u$). 

By work of Fenley \cite{fenley2012ideal}, the flow space can by compactified with a circle, called \defn{the boundary of the flow space} and denoted $\partial\orb$. This circle can be constructed by way of the foliations $\orb^{s/u}$ as follows. Consider the set 
\[\mc C=\{\text{ends of leaves of }\orb^{s/u}\}.\] This set has a natural circular order by considering the intersections of rays in $\orb^{s/u}$ with the boundaries of larger and larger closed disks in $\orb$. As described in \Cref{sec:circle_background}, $\mc C$ can be completed to a circle, which is what we call $\partial\orb$. The topology on $\orb\cup\partial\orb$ is defined so that leaves in $\orb^{s/u}$ limit onto the corresponding ends in $\partial\orb$. We use $\overline{\orb}$ to denote $\orb\cup\partial\orb$. The action $\pi_1(M)\curvearrowright\orb$ extends continuously to an action $\pi_1(M)\curvearrowright\overline{\orb}$.

\subsubsection{Structures in the flow space} \label{sec:struct}
Leaves $\ell^s$ and $\ell^u$ of $\orb^s$ and $\orb^u$ make a \defn{perfect fit} if there is an embedded rectangle \[R=i([0,1]\times[0,1])\subseteq\overline{\orb}\] such that
\begin{enumerate}
	\item $i(x,y)\in\partial\orb$ if and only if $(x,y)=(0,0)$,
	\item for all $x>0$, $i(x,0)\subseteq\ell^s$,
	\item for all $y>0$, $i(0,y)\subseteq\ell^u$, and
	\item the interior of $R$ is trivially foliated by $\orb^{s/u}$.
\end{enumerate}
See \Cref{fig:perfect_fit}. The rectangle $R$ is called a \defn{perfect fit rectangle}. 

\begin{figure}[h]
	\centering
	\includegraphics[width=0.8\linewidth]{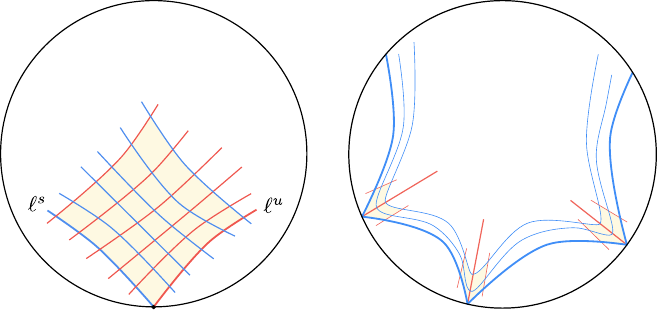}
	\caption{On the left, a perfect fit rectangle in yellow. On the right, pairwise nonseparated leaves of $\orb^{s}$ form a chain and force adjacent perfect fit rectangles, in yellow.}\label{fig:perfect_fit}
\end{figure}

Leaves of $\orb^{s/u}$ that are nonseparated in the leaf space (i.e. fail to have disjoint neighborhoods) are called \defn{branching} or \defn{nonseparated leaves}. By \cite[Theorem B]{fenley1998structure}, sets of pairwise nonseparated leaves form a chain, and consecutive leaves in the chain make 
adjacent perfect fits, as in \Cref{fig:perfect_fit}. By \cite[Corollary F]{fenley1998structure}, such chains always have finite length for an Anosov flow on an atoroidal 3-manifold.

A leaf of $\orb^{s/u}$ that doesn't make any perfect fits is called a \defn{regular} leaf. Such a leaf is necessarily non-branching.

Since the foliations $\orb^{s/u}$ have been oriented,
given a pair of nonseparated leaves in $\orb^{s/u}$, it makes sense to say whether they are \defn{branching from above} or \defn{branching from below}. See \Cref{fig:branching}.

\begin{figure}[h]
	\centering
	\includegraphics[width=\linewidth]{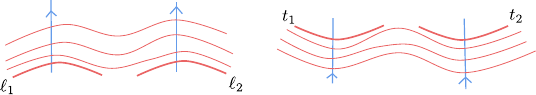}
	\caption{The unstable leaves $\ell_1$ and $\ell_2$ are branching from above, while the unstable leaves $t_1$ and $t_2$ are branching from below. The arrows on stable leaves indicate their orientations, which in turn determine the coorientations on intersecting unstable leaves.}\label{fig:branching}
\end{figure}

Fenley showed that if an Anosov flow with orientable 
foliations $W^{s/u}$ is non $\R$-covered, then it has two sided branching, i.e. both foliations have leaves that are branching from below, as well as leaves that are branching from above \cite{fenley1995sided}.

\subsubsection{``Good form'' for the flow}
\label{sec:goodform}

Following Barthelm\'e, Fenley, and Potrie \cite[Corollary 5.23]{barthelme2023collapsed} (see also \cite[Proposition 5.3]{potrie2025anosov}), we can replace $\varphi$ with an orbit equivalent (topological) Anosov flow whose unstable foliation $W^u$ has smooth leaves with the additional properties that:
\begin{enumerate}
\item there is a continuous metric $g$ on $M$ so that the restriction of $g$ to each unstable leaf of $W^u$ is hyperbolic,
\item orbits of $\varphi$ in each leaf of $W^u$ are geodesic, and 
\item the strong unstable foliations $W^{uu}$ in each leaf of $W^u$ are projections of horocycles.
\end{enumerate}

With these properties, we say the flow $\varphi$ is \defn{adapted} to $W^u$. Throughout the paper, we fix $\varphi$ to be the flow adapted to $W^u$.  This will be especially useful in characterizing the markers of $W^u$ as we will do below.

Similarly, there is an orbit equivalent flow that is adapted to $W^s$. We will also use this fact when carefully describing the `tilting' of $W^u$ in \Cref{sec:tilting}.

\section{Universal circles, markers, and the Calegari--Dunfield construction}
\label{sec:cd}

In this section, we define the notion of a universal circle for a foliation and review both the Calegari--Dunfield construction as well as the Landry--Minsky--Taylor construction (in the presence of a transverse pseudo-Anosov flow). Much of this involves explaining the tools used to define the Calegari--Dunfield universal circle, and we refer the reader to \cite{CalDun_UC} for additional details.

\subsection{The circle bundle at infinity and universal circles}\label{sec:ucs}

We begin by reviewing the structure associated to a taut foliation that is needed to define a universal circle.

\begin{figure}[h]
	\centering
	\includegraphics[width=0.55\linewidth]{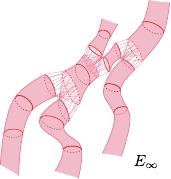}
	\caption{A piece of $E_\infty$ for a non $\R$-covered foliation. The intersticial regions indicate the topology around branching leaves. Selected circle fibers are in red.}\label{fig:Einfinityalt}
\end{figure}

Suppose $\mc F$ is a taut foliation of $M$. By a theorem of Candel \cite{Candel_uniformization}, we can put a \defn{leafwise hyperbolic metric} on $M$, i.e. a continuous Riemannian metric $g$ such that the pullback to each leaf is hyperbolic. Lifting everything to the universal cover $\wt M$, we have that $\wt M\cong \R^3$, and each leaf of $\wt{\mc F}$ is topologically a plane. 

Let $\mc L$ denote the leaf space of $\mc F$, i.e. $\mc L=\wt M/\widetilde{\mc F}$.
Recall that if $\mc L\cong\R$, we say $\mc F$ is \defn{$\R$-covered}, otherwise it is \defn{non $\R$-covered}. For a non $\R$-covered foliation, two leaves of $\mc L$ are \defn{comparable} if there is a transversal in $\wt M$ connecting them and \defn{incomparable} otherwise. Two incomparable leaves are in the same \defn{cataclysm} if they are nonseparated as points in $\mc L$.

Geometrically, each leaf $\lambda$ of $\wt{\mc F}$ is isometric to $\HH^2$ under the pullback of the lifted metric $\wt g$. Thus, we can compactify $\lambda$ with a circle $\partial_\infty\lambda$ via the standard compactification of $\HH^2$ with $\partial_\infty\HH^2$. As a set, we define the \defn{circle bundle at infinity} by \[E_\infty=\bigcup_{\lambda\in\mc L}\partial_\infty\lambda.\] We call the circles $\partial_\infty\lambda$ the \defn{circle fibers}.

We topologize this set as follows. Let $\tau$ be a transversal to $\wt{\mc F}$, i.e. an embedded interval in $\wt M$ transverse to $\wt{\mc F}$. Then $\tau$ determines an embedded interval in $\mc L$, which we also call $\tau\subset \mc L$. We define a map \[\text{UT}\wt{\mc F}\mid_\tau\to E_\infty\mid_\tau\] by sending a point $(p,v)$ to the forward endpoint of the geodesic through $p$ with tangent vector $v$. We then topologize $E_\infty$ such that this map is a homeomorphism for all $\tau$. See \cite{CalDun_UC} for the details and \Cref{fig:Einfinityalt} for an illustration.

With this topology, $E_\infty$ has the structure of a circle bundle over the leaf space, where the fibers are the various circles at infinity for leaves of $\wt{\mc F}$. This space inherits an action of $\pi_1(M)$ coming from the action by deck transformations on $\wt M$. 

We are now ready to define the notion of a universal circle for $\mc F$. 

\begin{definition}\label{def:uc}
	A \defn{universal circle} for $\mc F$ is a circle $\mc C$ together with an action $\pi_1(M)\curvearrowright\mc C$ and a set of monotone maps \[\{\pi_\lambda\colon\mc C\rightarrow \partial_\infty\lambda\:\mid\:\lambda\in\mc L\}\] such that the following conditions are satisfied:
	\begin{enumerate}
		\item For all $g\in\pi_1(M)$ and for all $\lambda\in\mc L$, the following diagram commutes, where the horizontal maps come from the respective actions $\pi_1(M)\curvearrowright\mc C$ and $\pi_1(M)\curvearrowright E_\infty$.
		\[\begin{tikzcd}
			{\mc C} & {\mc C} \\
			{\partial_\infty\lambda} & {\partial_\infty g\cdot\lambda}
			\arrow["{\cdot g}", from=1-1, to=1-2]
			\arrow["{\pi_\lambda}"', from=1-1, to=2-1]
			\arrow["{\pi_{g\cdot\lambda}}", from=1-2, to=2-2]
			\arrow["{\cdot g}"', from=2-1, to=2-2]
		\end{tikzcd}\]
		\item The function $\pi\colon\mc C\times\mc L\to E_\infty$ defined by \[\pi(z,\lambda)=\pi_\lambda(z)\] is continuous.
		\item For any pair $\lambda$, $\mu$ of incomparable leaves of $\mc L$, the core of $\pi_\lambda$ is contained in a single gap of $\pi_\mu$. \qedhere
	\end{enumerate} \end{definition}

\smallskip	
The universal circle $\mc C$ is \defn{minimal} if no open interval in $\mc C$ is contained in a gap of $\pi_\lambda$ for all $\lambda \in \mc L$. Equivalently, there are no distinct points $x,y \in \mc C$ such that $\pi_\lambda(x) = \pi_\lambda(y)$ for each $\lambda$. Indeed, if such points are always identified, then each $\pi_\lambda$ either collapses $[x,y]$ or $[y,x]$ as oriented subintervals of $\mc C$, and we can partition $\mc L$ into $\mc L_{[x,y]}$ or $\mc L_{[y,x]}$ accordingly. These are disjoint closed sets whose union is $\mc L$ and so one of them must be empty. Thus, we can collapse such intervals to build a minimal universal circle.

\smallskip

For each point $z\in\mc C$, we can define a (continuous) section $\sigma_z \colon \mc L \to E_\infty$ of $E_\infty$, as a circle bundle over the leaf space, by 
\[
\sigma_z(\lambda) = \pi_\lambda(z).
\]
The set of sections for a given universal circle can coalesce, but they can't cross, essentially because the maps $\pi_\lambda$ are monotone. 

Given two universal circles $\mc C$, $\mc C'$ with monotone maps $\{\pi^{\mc C}_\lambda\}_\l,\{\pi_\lambda^{\mc C'}\}_\l$, we will say they are \defn{isomorphic} as universal circles, denoted $\mc C\cong\mc C'$, if there is a $\pi_1(M)$-equivariant homeomorphism $h\colon \mc C\rightarrow\mc C'$ that intertwines the monotone maps, i.e. such that \[\pi_\lambda^{\mc C}=\pi_\lambda^{\mc C'}\circ h\] for all $\lambda\in\mc L$. By the following lemma, isomorphic universal circles determine the same sets of sections, so this is a rather strong notion of isomorphism.

\begin{lemma}\label{lem:same_sections}
	Let $\mc C_1,\mc C_2$ be universal circles for $\mc F$ with sets of sections $\mc S_1$ and $\mc S_2$. If $\mc C_1\cong\mc C_2$, then $\mc S_1=\mc S_2$. The converse is true if $\mc C_1$ and $\mc C_2$ are minimal.
\end{lemma}

\begin{proof}
	Let $\{\pi^i_\lambda\:\mid\:\lambda\in\mc L\}$ be the monotone maps for $\mc C_i$, and let $h\colon\mc C_1\rightarrow\mc C_2$ be the conjugating homeomorphism. Let $z\in \mc C^1$ determining a section $\sigma^1_z$. Then $\sigma^2_{h(z)}$ is a section of $\mc C_2$, and for all $\lambda\in\mc L$, we have  \[\sigma^1_z(\lambda)=\pi^1_\lambda(z)=\pi^2_\lambda\circ h(z)=\sigma^2_{h(z)}(\lambda).\] Thus $\sigma^1_z=\sigma^2_{h(z)}$.
	
	For the converse, note that a universal circle $\mc C$ is minimal if and only if the map $z\mapsto \sigma_z$ from $\mc C$ to the space of all sections of $E_\infty \to \mc L$ is injective. The image is the set $\mc S$ of sections of $\mc C$, and $\mc C$ can be recovered from $\mc S$ by taking the induced circular order. From this, the required statement follows.
\end{proof}

A weaker notion of equivalence is that of a conjugate group action -- part of the data of a universal circle is an action of $\pi_1(M)$ on a circle, and one could ask whether two universal circles give rise to conjugate actions. This is the question we address for $W^u$ in  \Cref{sec:nonconjugate}.

\subsection{Markers and the Calegari--Dunfield construction}
\label{sec:CD}
To describe the Calegari--Dunfield construction, we first need to define \emph{markers} in $E_\infty$. The basic idea behind markers is that they track directions in which leaves of $\mathcal{ \widetilde{F}}$ stay bounded distance as they go out to infinity.

The \defn{separation constant} for $\mc F$ is the largest number $\delta$ such that each leaf of $\mc F$ is quasi-isometrically embedded in its $\delta$-neighborhood (using the Riemannian metric $g$ as above). For any $\varepsilon<\delta$, and \defn{$\varepsilon$-marker} is an embedding $m\colon [0,1]\times [0,\infty)\to \wt M$ such that the following properties are satisfied.

\begin{enumerate}
	\item For each $t\in[0,1]$, $m(t,[0,\infty))$ is contained in a single leaf of $\widetilde{\mc F}$, and is a geodeisc ray in that leaf. These are called \defn{horizontal rays}.
	\item For each $x\in[0,\infty)$, $m([0,1],x)$ is a transversal of height less than $\varepsilon/3$.
\end{enumerate}

Since each horizontal ray in a marker is a geodesic in a leaf, it limits to a point on the boundary of that leaf. Thus, markers limit onto embedded intervals in $E_\infty$, and we will also call these \defn{markers}. A point in $E_\infty$ contained in a marker will be called a \defn{marker point}.

Markers in $E_\infty$ can be amalgamated by joining them at their endpoints, and doing this always yields an embedded interval transverse to the circle fibers \cite[Lemma 6.11]{CalDun_UC}.  Abusing notation, we call a maximal such amalgamation a \defn{maximal marker}, and consider it a marker (though it may be too tall to be a true marker). A key property of markers is that they cannot cross.

A section $s\colon\mc L\to E_\infty$ is called \defn{admissible} if it does not cross markers. Given a point $p\in E_\infty$, we define the \defn{leftmost section based at $p$}, denoted $s^\ell_p$, as follows. Let $\tau$ be a line in $\mc L$ such that $p\in E_\infty\mid_\tau$. In $E_\infty\mid_\tau$ and traveling up from $p$, $s^\ell_p$ is defined to be the section that is `leftmost' among all admissible sections through $p$. Traveling down from $p$ over $\tau$, $s^\ell_p$ is the section that is `rightmost' among admissible sections through $p$. This is made precise by approximating $s^\ell_p\mid_\tau$ by paths $\gamma_i$ that travel left above $p$ and right below, except when it would mean crossing a marker in the set $M_i$, where $M_i\subseteq M_{i+1}$, and $\cup_i M_i$ is the set of all markers. See \Cref{fig:approximating_slp} for the idea, and \cite{CalDun_UC} for details.

\begin{figure}[h]
	\centering
	\includegraphics[width=0.9\linewidth]{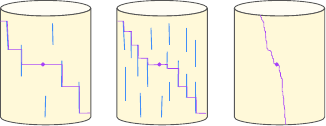}
	\caption{Approximating the leftmost section based at $p$. Markers are drawn in blue, and the highlighted point is $p$. On the left and in the middle are successive aproximations, and on the right is $s^\ell_p$, which looks something like a sideways graph of the Cantor function.}\label{fig:approximating_slp}
\end{figure}

To extend $s^\ell_p$ to the rest of $\mc L$, we use the `turning corners' rule. Suppose we've defined $s^\ell_p$ at a leaf $\lambda$, and need to extend it to a leaf $\mu$ nonseparated from $\lambda$. By Lemma 6.18 of \cite{CalDun_UC}, there is a canonical `monotone relation' between $\partial_\infty \lambda$ and $\partial_\infty \mu$. Effectively, this means that there are canonical points $p_{\lambda\mu} \in \partial_\infty \mu$ and $p_{\mu\lambda}\in \partial_\infty \lambda$ with the property that if $\{\lambda, \mu\} \subset C \subset \mc L$ is path connected and $s \colon C \to  E_\infty|_C$ is an admissible section, then $(1)$ $s(\mu) \neq p_{\lambda\mu}$ implies $s(\lambda) = p_{\mu\lambda}$, and similarly $(2)$ $s(\lambda) \neq p_{\mu\lambda}$ implies $s(\mu) = p_{\lambda\mu}$.

Returning to the definition of $s^\ell_p$, we set 
\[
s^\ell_p(\mu)= p_{\lambda\mu}
\] 
and then continue to define $s^\ell_p$ using the leftmost up, rightmost down rule from $s^\ell_p(\mu)$. Proceeding this way defines a section on all of $\mc L$.

\smallskip
Let 
\[
\mc S=\{s^\ell_p\;\mid\;p\in E_\infty\}.
\] 
The circular orders on $\partial_\infty\lambda$ over various $\lambda\in\mc L$ are compatible along the sections of $\mc S$, and they thereby induce 
a well-defined circular order on $\mc S$. Furthermore, the action by $\pi_1 (M)$ on $E_\infty$ gives an action on $\mc S$ that preserves the circular order. Completing the order and collapsing gaps as described in \Cref{sec:circle_background} yields a circle, $\mc C^\ell$, called the \defn{leftmost universal circle}, with an action by $\pi_1 M$. Each point in $\mc C^\ell$ gives us a section, where limit sections are defined by pointwise convergence. Monotone maps \[\pi_\lambda^\ell\colon\mc C^\ell\to\partial_\infty\lambda\] are defined by evaluating sections on circle fibers. Thus, $\mc C^\ell$ is a universal circle for $\mc F$ in the sense of \Cref{def:uc}.

To define the \defn{rightmost universal circle} $\mc C^r$, we run the same construction, but starting with \defn{rightmost sections} $s^r_p$ that are rightmost up and leftmost down.

We will call a universal circle \defn{admissible} if its sections do not cross markers. By construction, $\mc C^\ell$ and $\mc C^r$ are admissible.

\subsection{The flowspace universal circle}\label{subsec:flowspace_uc_background}

Now suppose that $\mc F$ is a foliation of $M$ that is transverse to a pseudo-Anosov flow $\psi$ with invariant foliations $W^{s/u}$. 
Then by \cite{LMT_UC}, the boundary of the orbit space of $\psi$ defines a universal circle for $\mc F$ in the sense of \Cref{def:uc}. We describe the basic idea of the construction below assuming the flow is Anosov, which is the case of interest to us. We will give a more detailed description specialized to our setting (where the foliation $\mc F$ is a tilted version of the unstable foliation $W^u$) in \Cref{sec:tilting_uc}.

As described in \Cref{sec:anosov_background}, Fenley showed that the action of $\pi_1(M)$ on $\orb$ extends continuously to an action on $\orb\cup\partial\orb$. Referring to the definition of a universal circle, we take as the circle and action the restricted circle action $ \pi_1(M) \curvearrowright \partial \orb$.
What remains is to construct the monotone maps $\{\pi_\lambda\colon\partial\orb\rightarrow \partial_\infty\lambda\:\mid\:\lambda\in\mc L\}$.

\begin{figure}[h]
	\centering
	\includegraphics[width=0.8\linewidth]{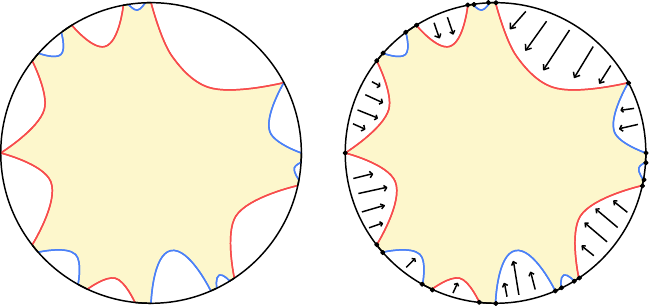}
	\caption{On the left: a typical shadow, in yellow. On the right: the map $h\colon\overline{\orb}\to\overline{\Omega}_\lambda$ used in the construction of the monotone map $\pi_\lambda\colon\partial\orb\rightarrow \partial_\infty\lambda$.}\label{fig:shadow_def}
\end{figure}

Let $\Theta\colon \wt M\to\orb$ denote projection to the flow space. 
A leaf $\lambda\in\mc L$ is a properly embedded plane in $\wt M$ intersecting each orbit of $\wt \phi$ at most once, hence the restriction $\Theta|_\lambda$ is an embedding onto $\lambda$'s \defn{shadow}:
 \[
 \Omega_{\lambda}=\Theta(\lambda)\subseteq \orb.
 \] 
Its frontier in $\orb$, denoted $\fr  \Omega_\lambda$ is a union of leaves of $\orb^{u/s}$ \cite{Fen09}. 
See \Cref{fig:shadow_def}. Let $\overline{\Omega}_\lambda$ be the closure of $\Omega_\lambda$ in $\overline{\orb}$. Then 
\[
\partial\overline{\Omega}_\lambda =\fr\Omega_\lambda\cup\partial_\infty\Omega_\lambda,
\]
 where $\partial_\infty\Omega_\lambda=\partial\orb\cap\partial\overline{\Omega}_\lambda$.

The leaf $\lambda$ has a pair of transverse foliations $\lambda^{s/u}$ by intersecting with $\wt W^{s/u}$, and $\Omega_\lambda$ has foliations $\Omega_\lambda^{s/u}$ by intersecting with $\orb^{s/u}$. Furthermore, the projection map $\Theta$ sends $\lambda^{s/u}$ to $\Omega_\lambda^{s/u}$. Mapping endpoints of rays in $\Omega_\lambda^{s/u}$ to endpoints of rays in $\lambda^{s/u}$ determines a circular order preserving map that can be uniquely extended to a monotone map $f \colon \partial \ol{\Omega}_\lambda \to \partial_\infty \lambda$, which collapses each leaf of $\fr  \Omega_\lambda$ to a single point of $\partial_\infty \lambda$ \cite[Section 3]{LMT_UC}.

To define the monotone map $\pi_\lambda\colon\partial\orb\to\partial_\infty\lambda$, we precompose $f$ with a map $h\colon\overline{\orb}\to\overline{\Omega}_\lambda$ that fixes $\partial_\infty\Omega_\lambda$, as in the right of \Cref{fig:shadow_def}. 

The action $\pi_1(M) \curvearrowright \partial \orb$ together with the monotone maps $\pi_\lambda\colon\partial\orb\rightarrow \partial_\infty\lambda$ determines a universal circle for $\mc F$. The relation between this universal circle and the one constructed in \Cref{sec:CD} is open in general, and we now turn to investigate the connection when $\mc F$ is the unstable foliation of an Anosov flow.

\section{Tilting, shadows, and the flow space boundary}\label{sec:tilting}

In this section, we isotope the foliation $W^u$ so that it becomes transverse to $\varphi$. We call this `tilted' foliation $W^u_\tilt$. A variation of this construction was first defined and studied by Fenley for $\RR$--covered foliations in \cite[Section 5]{fenley2005regulating}. Also see \cite[Section 5]{baik2024reconstruction} for a related construction in a more general setting. For us, it will be important to understand the shadows of leaves of $W^u_\tilt$ in terms of the original leaves of $W^u$, so we describe our `tilting' rather explicitly.

\subsection{Tilting $W^u$ to $W^u_\tilt$} \label{sec:tilting_fol}
To begin, recall that our flow $\varphi$ is fixed and in the form adapted to $W^u$. For the construction, we make the temporary switch to the orbit equivalent flow that is adapted to $W^s$. Since the stable/unstable foliations of these orbit equivalent flows are homeomorphic, we do not change our notation to make the distinction.  

Informally, we set $h_t$
to be the continuous flow on $M$ defined by moving each point along its leaf of $W^{ss}$ in the positive direction at unit speed. Here $W^{ss}$ is the strong stable foliation of $\varphi$ (i.e. the one-dimension foliation whose leaves consist of points that are asymptotic under the forward flow), the existence of which is guaranteed because the flow is adapted to $W^s$.
Then we will set $W^u_\tilt = h_1(W^u)$. 
To make this more precise (and to verify continuity) we work in the universal cover $\wt M$.

\smallskip
Recall that since $\varphi$ is adapted to $W^s$, each leaf $\lambda$ of the lifted foliation $\wt W^s$ is a copy of the hyperbolic plane with the lifted metric, and these metrics vary continuously.
Define $\wt h_t \colon \wt M \to \wt M$ to be the map defined on each leaf $\lambda$ by moving each point in $\lambda$ along the length $t$ arc of its strong stable leaf in the positive direction. We recall that the leaves of $\wt W^{ss}$ in $\lambda$ are precisely the horocycles of $\lambda$ based at the forward limit of orbits in $\lambda$. 
 
\begin{claim}\label{claim:strongstab}
The maps $\wt h_t \colon \wt M \to \wt M$ determine a continuous, equivariant flow. Hence, they descend to a continuous flow $h_t \colon M \to M$ that preserves the leaves of $W^s$ and has smooth orbits. 
\end{claim}

\begin{proof}
Since $W^u$ is cooriented, the leaves of $W^{ss}$ are oriented. Hence, the flow $h_t$ is obtained by parameterizing the oriented foliation $W^{ss}$ at unit speed.

Continuity follows from the properties of the adapted flow (established in \cite[Corollary 5.23]{barthelme2023collapsed}). First, the foliation $W^{ss}$ is continuous. Indeed, we have a continuous, leafwise hyperbolic metric $g$ on the leaves of $\wt W^s$, and in each leaf $\lambda$ of $\wt W^s$ the leaves of $\wt W^{ss}$ are 
horocycles based at the point $\theta_\lambda \in \partial_\infty \lambda$ that is the common forward endpoint of the orbits in $\lambda$. As the endpoint $\theta_\lambda$ varies continuously along leaves of $\wt W^s$ (as in \cite[Corollary 5.23]{barthelme2023collapsed}), so does the horocyclic foliation.

Finally, again using the continuity of $g$, the unit speed parameterization of the oriented leaves of $\wt W^{ss}$ defines a continuous flow of $\wt M$ that is smooth on each leaf of $\wt W^s$ (since the foliation by horocycles is smooth on the leaves of $\wt W^s$). 
\end{proof}

Finally, we let $h_t \colon M \to M$ be the induced flow on $M$, set 
$h = h_1$, and fix $W^u_\tilt=h(W^u)$. We call $W^u_\tilt$ the \defn{titled unstable foliation}.

\smallskip
The following description, which is clear from the construction, in particular shows that $W^u_\tilt$ is transverse to $\phi$. 

\begin{lemma} \label{lem:tilted_trans}
For each leaf $\lambda$ of $\wt W^s$, the map $\wt h_1$ preserves $\lambda$ and takes each geodesic orbit $\gamma$ in $\lambda \cap \wt W^u$ to a leaf $h(\gamma)$ of the foliation $\lambda \cap W^u_\tilt$ with the property that $h(\gamma)$ is the `positive side' of the boundary of the $1$--neighborhood of $\gamma$ in $\lambda$ (with respect to the orientation of $\lambda$ and its induced hyperbolic metric). 
\end{lemma}

See \Cref{fig:stableleaf}. 

We remark that $h \circ \phi_t \circ h^{-1}$ is a conjugate Anosov flow on $M$ whose stable foliation is $W^s$ and whose unstable foliation is $W^u_{\tilt}$. Its orbits, lifted to $\wt M$, are described exactly as in \Cref{lem:tilted_trans} and therefore uniformly fellow travel the orbits of $\wt \varphi$. 

\begin{figure}
	\centering
	\includegraphics[width=0.8\linewidth]{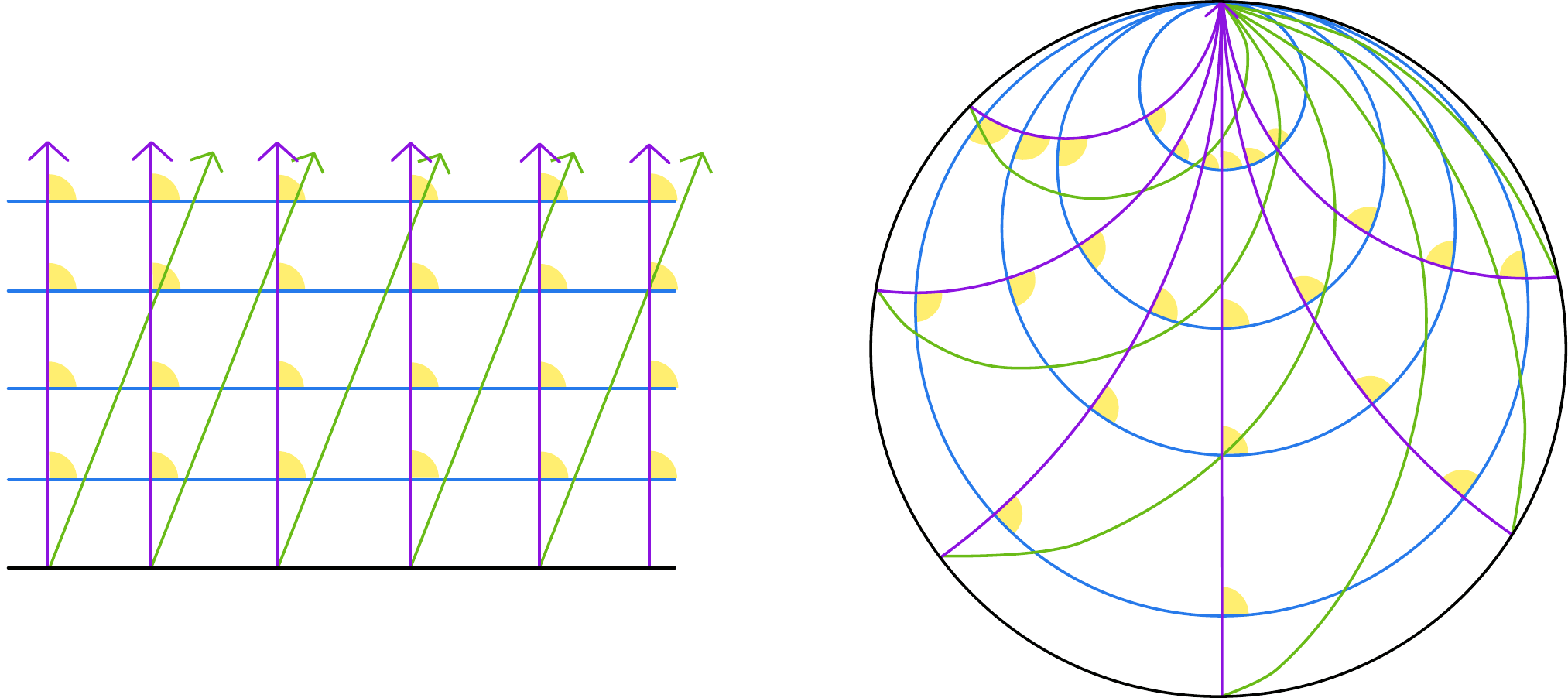}
	\caption{A lifted weak stable leaf $\lambda$ with flow lines of $\varphi$ in purple, the strong stable foliation in blue, and flow lines of the tilted flow $\varphi_{\tilt}$ in green. The yellow sectors contain $\frac{d}{dt}\varphi_{\tilt}$.}
	\label{fig:stableleaf}
\end{figure}

\begin{remark}[Tilting the other way]
The transverse foliation $W^u_\tilt = h(W^u)$ was produced using the flow $h_t \colon M \to M$ and setting $h = h_1$. In what follows, the choice of $1$ among $t>0$ is immaterial; this is because any two such foliations are isotopic through foliations \emph{transverse} to $\varphi$.

 However, we can also define the foliation $h_{-1}(W^u)$, which after reversing its coorientation is also positively transverse to $\varphi$. This foliation is \emph{not} isotopic to $W^u_\tilt$ through {transverse} foliations, and so it will have different properties (see e.g. \Cref{lem:leaf_shadow}). 
 
 In order to talk about both transverse foliations simultaneously, we  simplify the notation by setting $h^\pm = h_{\pm 1}$ and  $\wt h^\pm = \wt h_{\pm 1}$. Hence, in what follows, the lifts of the tilted foliations to $\wt M$ will be denoted $\wt h^\pm(\wt W^u)$.
\end{remark}

\subsection{Shadows of $W^u_\tilt$}
Now we describe the shadows of leaves of the tilted foliations $\wt h^\pm(\wt W^u)$, i.e. their projections to the flow space under the map $\Theta\colon\wt M\rightarrow\orb$. To set notation, for $\lambda$ a leaf of $\wt W^u$, we define
\[\Omega^\pm_\lambda=\Theta\circ\wt h^\pm(\lambda).\]
We will call $\Omega^+_\lambda$ the \defn{positive shadow} of $\lambda$ and $\Omega^-_\lambda$ the \defn{negative shadow}.

Let $\lambda$ be a leaf of $\wt W^u$. As $\wt h^\pm(\lambda)$ is foliated by flowlines of the tilted flow $\wt h^\pm\circ \wt\phi\circ (\wt h^\pm)^{-1}$, we first find the shadow of a single orbit of the tilted flow.

Given a point $p\in\orb$, we denote the leaf of $\orb^s$ through $p$ by $\orb^s(p)$. Given a set $X\subseteq\orb$, the \defn{positive stable saturation} of $X$ is
\[\orb^s_+(X)=\{q\in\orb^s(p)\text{ on the positive side of }p\:\mid\:p\in X\},\] where `on the positive side' is defined using the orientation on $\orb^s$. The \defn{negative stable saturation} is similarly defined as
\[\orb^s_-(X)=\{q\in\orb^s(p)\text{ on the negative side of }p\:\mid\:p\in X\}.\]

\begin{lemma} \label{lem:orbit_shadow}
	Let $\gamma$ be a flow line of $\varphi$ in $\lambda$. Then \[\Theta\circ\wt h^+(\gamma)=\orb^s_+(\Theta(\gamma))\] and
	\[\Theta\circ\wt h^-(\gamma)=\orb^s_-(\Theta(\gamma)).\]
\end{lemma}

\begin{proof}
	Let $\mu$ be the leaf of $\wt W^s$ containing $\gamma$. By \Cref{lem:tilted_trans}, we see that $\wt h^+(\gamma)$ intersects exactly the orbits of $\varphi$ lying within $\mu$ and on the positive side of $\gamma$. Similarly, $\wt h^+(\gamma)$ intersects exactly the orbits on the negative side of $\gamma$.
\end{proof}

As $\wt h^\pm(\lambda)$ is a union of flow lines of the tilted flow, we can combine their shadows to find the shadow $\Omega^\pm_\lambda$.

\begin{lemma}\label{lem:leaf_shadow}
	The positive and negative shadows of $\lambda$ are given by 
	\[\Omega^\pm_\lambda=\orb^s_\pm(\Theta(\lambda)).\]	
\end{lemma}

See \Cref{fig:shadow}.

\begin{figure}
	\centering
	\includegraphics[width=0.35\linewidth]{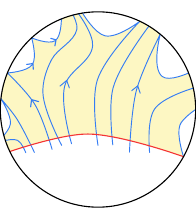}
	\caption{The red leaf is $\Theta(\lambda)$, and the yellow region is $\Omega^+_\lambda$. By the notation convention described below, $\Theta(\lambda)$ would be called $\ell$.}
	\label{fig:shadow}
\end{figure}

\begin{proof}
We have 
\[\Omega^\pm_\lambda=\{\Theta\circ\wt h^\pm(\gamma): \gamma \subseteq \lambda\text{ an orbit of }\varphi\},\] so combining the shadows described in \Cref{lem:leaf_shadow} produces the shadow described above.
\end{proof}

We set the following notational conventions that will be in use for the remainder of the paper. Greek letters such as $\lambda$ and $\mu$ will be used to represent leaves of $\widetilde{W}^u$ and $\widetilde{W}^s$, sometimes using superscripts of `$u$' and `$s$' respectively. Given a leaf $\lambda$ of $\widetilde{W}^{s/u}$, we will use the corresponding character in the Latin alphabet (in this case $\ell$) to denote $\Theta(\lambda)$, using the same subscript and superscript where applicable. For example, given a leaf $\lambda_i^u$ of $\widetilde{W}^u$, the frontier of $\Omega^\pm_{\lambda^u_i}$ always contains the leaf $\ell^u_i$ by \Cref{lem:leaf_shadow}.

\subsection{The flowspace universal circles for $W^u$} \label{sec:tilting_uc}
As a result of the tilting construction, we have that $\varphi$ is transverse to the tilted foliations $h^\pm(W^u)$. Thus, as in \Cref{subsec:flowspace_uc_background}, $\partial\orb$ gives us a universal circle for both $h^+(W^u)$ and $h^-(W^u)$.

However, each of these universal circles can be viewed as a universal circle for (the isotopic) foliation $W^u$ since $h^\pm$ induces a homeomorphism from $\mc L^u$ to the leaf spaces of $h^\pm(W^u)$ with the property that $\partial_\infty \lambda$ is identified with $\partial_\infty \wt h^\pm(\lambda)$, for each $\lambda \in \mc L$. This identification is canonical because $\wt h^\pm$ moves points bounded distance within leaves of $\wt W^s$. From this, it is easy to assemble monotone maps $\pi_\lambda^\pm \colon \partial \orb^\pm \to \partial_\infty \lambda$ so that 
$(\partial\orb^\pm,\{\pi^+_\lambda\})$ define universal circles for $W^u$. 

From now on, $\partial\orb^\pm$ will refer to these universal circles, i.e. the universal circles for $W^u$. Even though these universal circles are now associated to the same foliation, we keep the notation $\partial\orb^\pm$ to distinguish them. As we will show in \Cref{thm:distinct_ucs}, they are in fact nonisomorphic universal circles.  Given a point $z\in\partial\orb$, we denote the section of $E_\infty$ associated to $z$ by 
\[
\sigma^\pm_z\colon \mc L^u\to E_\infty,
\]
where $\sigma^\pm_z(\l)=\pi^\pm_\l(z)$.

We conclude by observing some properties of the monotone maps associated to these universal circles. We describe the situation for $\partial\orb^+$, but $\partial\orb^-$ is completely analogous. Let $\lambda$ be a leaf of $\wt W^u$. One frontier leaf of $\Omega^+_\lambda$ is the unstable leaf $\ell = \Theta(\lambda)$. The frontier chain containing $\ell$ also contains any leaves of $\orb^s$ making perfect fits with $\ell$ and on the positive side of $\ell$, plus possibly a nonseparated chain of leaves in $\orb^s$, as in \Cref{fig:shadow_gaps}.

The other frontier chains are chains of leaves in $\orb^s$, nonseparated from leaves of $\orb^s$ in $\Omega^+_\lambda$, and on the positive side of these leaves. Thus, $\ell$ is the only unstable leaf in $\fr\Omega^+_\lambda$.

\begin{figure}[h]
	\centering
	\includegraphics[width=0.8\linewidth]{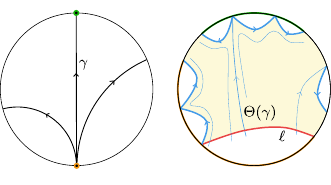}
	\caption{On the left is the leaf $\wt h^+(\lambda$), with the common backwards endpoint of its flow lines in orange, and the forward endpoint of a flowline $\gamma$ in green. On the right is $\orb$, with the shadow $\Omega^+_\lambda$ in yellow. Frontier chains are bolded, and the green and orange intervals in $\partial\orb$ are sent by 
	$\pi^+_\lambda$ to the respective points in $\partial_\infty \lambda$.} 
	\label{fig:shadow_gaps}
\end{figure}

The composition $\Theta\circ \wt h^+$ maps $\lambda$ homeomorphically to $\Omega^+_\lambda$. The foliation of $\lambda$ by flow lines is $\wt W^s \cap \lambda$ and maps to the foliation $\orb^s \cap \Omega^+_\lambda$. All leaves of $\orb^s\cap\Omega^+_\lambda$ limit onto $\ell$ on one side, so the frontier chain containing $\ell$ is mapped by $\pi^+_\lambda$ to the point $\theta_\lambda \in \partial_\infty \lambda$ that is the common backwards limit point of all flow lines in $\lambda$

All other frontier chains map to different points in $\partial_\infty \lambda$. A leaf $\ell^s$ of $\orb^s$ intersecting $\Omega^+_\lambda$ has a positive limit point $z$ in $\partial\orb$ contained in $\partial_\infty\Omega^+_\lambda$. This point is mapped to the forwards endpoint of the corresponding tilted flow line in $\lambda$, along with any leaves nonseparated from $\ell^s$ on the positive side.

To summarize this discussion:

\begin{lemma}\label{lem:gaps}
	The frontier chain of $\Omega^\pm_\lambda$ containing the unique unstable leaf $\ell$ is mapped to $\theta_\lambda\in\partial_\infty\lambda$ by $\pi^\pm_\lambda$, and all other frontier chains are mapped to forward limits of orbits. 
\end{lemma}

\section{Parameterizing the circle bundle at infinity}
\label{sec:para}
To understand the leftmost and rightmost universal circles for $W^u$, we need a much better understanding of the structure of $E_\infty$, and in particular what the set of markers looks like. In this section, we describe how to obtain $E_\infty$ by gluing together a copy of $\orb$ with a copy of $\mc L^u$ in a way that captures the set of markers as well as the $\pi_1(M)$ action.

\subsection{Markers in $E_\infty$}
\label{sec:markers}
For a leaf $\lambda$ of $\wt W^u$, let $\theta_\lambda \in \partial_\infty \lambda$ be the common backward endpoint of all orbits in $\lambda$. The function $\mc L \to E_\infty$ mapping $\lambda$ to $\theta_\lambda$ is known to be continuous (as in the proof of \Cref{claim:strongstab}), but this will also follow from observations below.

\begin{lemma}\label{lem:stab_marker}
	Leaves of $\wt W^s$ define $\epsilon$--markers for $E_\infty$, for every $\epsilon >0$. Moreover, these markers pass through every point of $\partial_\infty \lambda \ssm \theta_\lambda$ for $\lambda \in \mc L^u$.
\end{lemma}

\begin{proof}
Pick $\eta \in \partial_\infty \lambda \ssm \theta_\lambda$. Then $\eta$ is the forward endpoint of an orbit $\gamma$ in $\lambda$ and we take $\mu$ to be the leaf of $\wt W^s$ such that $\lambda \cap \mu = \gamma$. Let $\tau$ be a small arc of $\mu$ crossing $\gamma$ and let $m$ the positive saturation of $\tau$ under the flow. More precisely, $m \colon \tau \times [0,\infty) \to \wt M$ is the embedding given by $(x, t) \mapsto \phi_t(x)$.
The map $m$ is a marker: for each $x \in \tau$, $m(x,[0,\infty))$ is a geodesic ray in a leaf of $\wt W^u$ since the flow is adapted to $\wt W^u$, and the height of each transversal is decreasing since $m$ consists of forward orbits in a stable leaf (where the holonomy is contracting). This completes the proof.
\end{proof}

\begin{lemma}\label{lem:nonmarker}
	There is an $\epsilon>0$ so that for each leaf $\l$ of $\wt W^u$, $\theta_\l$ is the unique point in $\partial_\infty\l$ that is not in the image of an $\epsilon$--marker, and this point varies continuously in $\mc L^u$. 
\end{lemma}

\begin{proof}
We first show that $\theta_\lambda$ is not in the image of a marker. More precisely, let $\epsilon$ be smaller than the separation constant of $\wt W^s$ and have the property that if $p\in \lambda$ is distance at most $\epsilon$ from some leaf $\lambda'$ of $\wt W^u$, then the leaf $\nu$ of $\wt W^s$ through $p$ also intersects $\lambda'$. We show that $\theta_\lambda$ is not in the image of an $\epsilon/2$--marker.

Suppose not, and let $m$ be an $\epsilon/2$--marker through $\theta_\lambda$ based at an geodesic ray $\gamma$ in $\lambda$. 
Note that $\gamma$ is also a (backwards) orbit in $\lambda$. 
The marker $m$ determines an arc $\partial m\subseteq E_\infty$ (also called a marker) that does not cross the markers coming from the stable foliation, as in \Cref{lem:stab_marker}. Here, we essentially aim to prove that all markers come from the stable foliation, but for the purposes of this proof, we call such markers in $E_\infty$ \defn{stable arcs}.

Since $\partial m$ does not cross any stable arcs, and such arcs pass through all points that are not of the form $\theta_\lambda$ by \Cref{lem:stab_marker}, then up to replacing $m$ with a smaller maker and possibly replacing $\lambda$, either $\partial m$ consists entirely of points of the form $\theta_\lambda$ or there is an open stable arc $a$ that agrees with $\partial m \ssm \theta_\lambda$. 

In the first case, let $\nu$ be the leaf of $\wt W^s$ whose intersection with $\lambda$ contains backwards orbit $\gamma$. Note that the backwards orbit $\gamma'$ obtained by intersecting $\nu$ with any leaf $\lambda'$ has its endpoint at $\theta_{\lambda'}$. Hence, the arc $a$ determined by $\nu$ (starting at $\theta_\lambda)$ agrees with the marker $\partial m$. This implies that, for each leaf $\lambda'$ sufficiently close to $\lambda$ and meeting $m$, the geodesic rays $m \cap \lambda'$ and $\nu \cap \lambda'$ are asymptotic. Hence, up to shortening the maker $m$, we can suppose that they have distance less than $\epsilon/4$ in each such $\lambda'$. By the Anosov dynamics, the distance \emph{in} the leaf $\nu$ between the backward orbits $\nu \cap \lambda$ and $\nu \cap \lambda'$ goes to $\infty$. On the other hand, in $\wt M$, one can pass from $\nu  \cap \lambda'$ to $m \cap \lambda'$, then to $\lambda$ along $m$, then back to $\nu$ in $\lambda$. This shows that the distance between the orbits $\nu \cap \lambda'$ and $\nu \cap \lambda$ is at most $\epsilon$ in $\wt M$, contradicting that $\nu$ is quasi-isometrically embedded in its $\epsilon$ neighborhood.

In the second case, let $\nu$ be the leaf of $\wt W^s$ that induces a marker containing $a$. Note that since we are assuming that $a$ is a marker, the points in $a$ correspond to forward orbits in $\nu \cap \lambda'$.
Since $a$ agrees with $\partial m$, and $m$ is an $\epsilon/2$--marker, $\nu$ must also intersect $\lambda$ and by continuity, we must have that $\lambda \cap \nu$ has its forward orbit $\theta_\lambda$. But we also know that the backward orbit of $\lambda \cap \nu$ is also $\theta_\lambda$, contradicting that the orbit is a geodesic in $\lambda$.

It only remains to prove that $\lambda \mapsto \theta_\lambda$ is continuous. For this, fix $\lambda_0 \in \mc L$ and consider $x \in \partial \orb$ that is contained in the interior of the gap 
$\pi_{\lambda_0}^{-1}(\theta_{\lambda_0})$ with the property that $x$ is the endpoint of a leaf of $\orb^s$ that crosses $\Theta(\lambda_0)$.
Then $x$ is also contained in the gap $\pi_{\lambda}^{-1}(\theta_{\lambda})$ for leaves $\lambda$ sufficient close to $\lambda_0$. Hence $\lambda \mapsto \theta_\lambda$ agrees with the (continuous) section $\sigma_{x}$ near $\lambda_0$. Since $\lambda_0$ was arbitrary, this proves the required continuity.
\end{proof}

With this in mind, we define the \defn{nonmarker section} $s_\textbf{nm}\colon \mc L\rightarrow E_\infty$ by \[s_\textbf{nm}(\l)=\theta_\l.\] 

\subsection{The stitching map $\Phi \colon \orb \to E_\infty$}
In this section, we introduce the key technical tool to understand the structure of markers on $E_\infty$; an identification of $E_\infty \ssm s_{\rm{nm}}$ with the flow space $\orb$ itself.

For this, we define the \defn{stitching map} $\Phi \colon \orb \to E_\infty$ as follows: First, let $\wt \Phi \colon \wt M \to E_\infty$ be the map that associates each $p \in \wt M$ contained in a leaf $\lambda$ of $\wt W^u$ to the endpoint in $\partial_\infty \lambda$ of the ray $(\phi_t(p))_{t\ge0} \subset \lambda$. This map is clearly equivariant and constant on orbits of $\varphi$, and we let $\Phi$ denote the induced map on $\orb$. 

\begin{theorem}\label{thm:stitching}
The stitching map $\Phi$ is onto the complement of the nonmarker section and induces an equivariant homeomorphism $\Phi \colon \orb \to E_\infty \ssm s_{\textbf{nm}}$. 

Further, $\Phi$ identifies the stable foliation $\orb^s$ of $\orb$ with the foliation of $E_\infty \ssm s_{\textbf{nm}}$ by markers and identifies the unstable foliation $\orb^u$ of $\orb$ with the foliation of $E_\infty \ssm s_{\textbf{nm}}$ by open arcs $\partial_\infty \lambda \ssm \theta_\lambda$ for $\lambda \in \mc L^u$.
\end{theorem}

\begin{proof}
For each unstable leaf $\lambda$, the only point of $\partial_\infty \lambda$ that is not the endpoint of a forward orbit of the flow is the unique nonmarker point. This shows that $\Phi \colon \orb \to E_\infty \ssm s_{\textbf{nm}}$ is surjective. Since no two forward orbits in $\lambda$ have the same endpoint in $\partial_\infty \lambda$, we also see that $\Phi$ is injective. Moreover, it follows that if $\lambda \in \mc L^u$ with $\ell=\Theta(\l)$, then $\Phi$ maps $\ell$ homomorphically onto $\partial_\infty \lambda \ssm \theta_\lambda$.

Next we show that $\Phi$ takes the stable foliation $\orb^s$ to the foliation of $E_\infty \ssm s_{\textbf{nm}}$ by markers. Fix a leaf $\mu$ of the stable foliation $\wt W^s$ and let $m$ denote its projection to a leaf of $\orb^s$. Unpacking the definitions, the image $\Phi(m)$ is obtained by, for each leaf $\lambda \in \mc L^u$ that intersects $\mu$, considering the forward endpoints of the orbits $\lambda \cap \mu$ in $\partial_\infty \lambda$. By \Cref{lem:stab_marker}, these are markers of $E_\infty$, and since they cover $E_\infty \ssm s_{\textbf{nm}}$, every marker is the image of such a leaf.

Using this structure, we now prove that $\Phi$ is a homeomorphism. Fix a stable leaf $m$ of $\orb^s$ and consider the restricted circle bundle $E_\infty|_{m}$ over $m$, which is an open cylinder in $E_\infty$. Here, we are also considering $m$ as a line in $\mc L^u$.
The preimage $\Phi^{-1}(E_\infty|_{m})$ is exactly the unstable saturation $\mc S^u (m)$ in $\mc \orb$, i.e. all points of $\orb$ whose unstable leaf crosses $m$. Since $\mc S^u (m)$ is open in $\orb$, it suffices to show that the restriction of $\Phi$ to ${\mc S^u (m)}$ is a homeomorphism onto $E_\infty|_{m}\ssm s_{\rm{nm}}$. Actually, by invariance of domain, we only need to show that the map is open.

\begin{figure}[h!]
	\centering
	\includegraphics[width=0.9\linewidth]{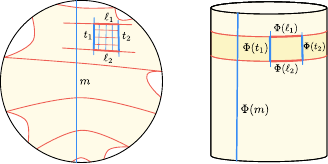}
	\caption{On the left is the open foliation box $U$ in $\orb$, with $\mc S^u (m)$ highlighted in yellow. On the right is $\Phi(\partial U)$ in $E_\infty\mid_{m}$, with $E_\infty\mid_{t_1}$ highlighted in a darker yellow.}\label{fig:openbox}
\end{figure}

For this, let $U\subset {\mc S^u (m)} \subset\orb$ be an open foliation box whose closure is bound by segments of leaves $\ell_1,\ell_2$ of $\orb^u$ and segments of leaves $t_1,t_2$ of $\orb^s$, as in \Cref{fig:openbox}. Note that $t_1$ determines an open embedded line in the projection of $m$ to $\mc L^u$, so that $E_\infty\mid_{t_1}$ is an open embedded cylinder in $E_\infty
|_{m}$. Then we know $\partial\overline{U}$ gets mapped to the closed nullhomotopic curve $\Phi(\partial\overline{U})\subseteq  E_\infty|_{m}$ as in \Cref{fig:openbox}, and $U$ must be mapped to the interior region since $\Phi$ is continuous on leaves of $\orb^u$. This shows that $\Phi(U)$ is open and completes the proof.
\end{proof}

Note that this means $E_\infty\ssm s_{\textbf{nm}}$ is Hausdorff, so as an open interval bundle over $\mc L$ it has no sections. In other words, we have the following corollary.

\begin{corollary}\label{cor:no_section}
	Every section of $E_\infty$ passes through $s_\mathbf{nm}$ at some point.
\end{corollary}

In light of \Cref{thm:stitching} and our conventions for drawing $\orb^{u/s}$, we will use the following convention when drawing $E_\infty$. Circle fibers (i.e. circles $\partial_\infty\lambda$ for $\lambda\in\mc L^u$) will be drawn in red, and will be horizontal. Markers will be drawn in blue, and the nonmarker section will be drawn in orange.

\subsection{Building $E_\infty$ over lines in $\mc L$}\label{subsec:over_lines}
In practice it is easier to work in cylinders in $E_\infty$, in which case it is clear to see how to construct the cylinder from $\orb$. 
Let $t$ be a transversal to $\orb^u$ in $\orb$ (possibly but not necessarily a leaf of $\orb^s$), and consider $\mc S^u(t)$, the unstable saturation of $t$.

Topologically, $\mc S^u(t)$ is an open rectangle that is trivially foliated by leaves of $\orb^u$, as in \Cref{fig:skew_ex} and \Cref{fig:nonR_ex}. 
We partially compactify $S^u(t)$ by adding two copies of $t$ as the endpoints of leaves of $\orb^u\mid_{\mc S^u(t)}$ -- note that this is \emph{not} done in $\overline{\orb}$. Then we identify the two copies of $t$ by gluing the opposite endpoints of each leaf of $\orb^u\mid_{\mc S^u(t)}$, obtaining the cylinder $E_\infty\mid_t$.

To see why this gives us $E_\infty\mid_t$, let $X$ be the cylinder we construct, let $s$ be the line of compactification points in $X$, and define 
\[
h\colon X\to E_\infty\mid_t
\] 
by $h(p)=\Phi(p)$ for $p\in X\ssm s$ and $h(q)=s_\textbf{nm}(\lambda)$ for $q\in s$ an endpoint of $\Theta(\lambda)$. Then $h$ is a homeomorphism onto its image when restricted to $X$ by \Cref{thm:stitching} and to $X\ssm s$ by construction. All we need to check is the continuity of $h$ and $h^{-1}$ in a neighborhood of $s$. This is easy to do using the fact that $X$ and $E_\infty\mid_t$ are both foliated by circles, and $h$ sends one foliation to another.

To illustrate, we will run this construction in the case that $\orb$ is skew.

\begin{example}[The skew case]
		\begin{figure}[h!]
			\centering
			\includegraphics[width=0.8\linewidth]{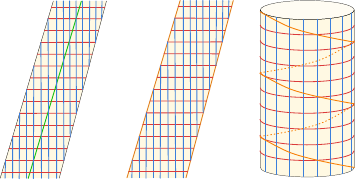}
			\caption{Constructing $E_\infty$ from $\orb$ in the skew case. On the left is $\mc S^u(t)\subseteq\orb$ with $t$ in green, in the middle is the partial compactification, and on the right is $E_\infty$.}\label{fig:skew_ex}
		\end{figure}
	
	Suppose $\orb$ is the skew plane, and pick a global transversal $t$ as in the left of \Cref{fig:skew_ex}. Taking the partial compactification adds the left and right sides of the skew plane, and by gluing opposite sides together we obtain all of $E_\infty$. We see that the nonmarker section twists around the cylinder, and markers are given by vertical segments. This recovers the well-known picture of $E_\infty$ in the skew case.
\end{example}

\begin{figure}[h!]
	\centering
	\includegraphics[width=\linewidth]{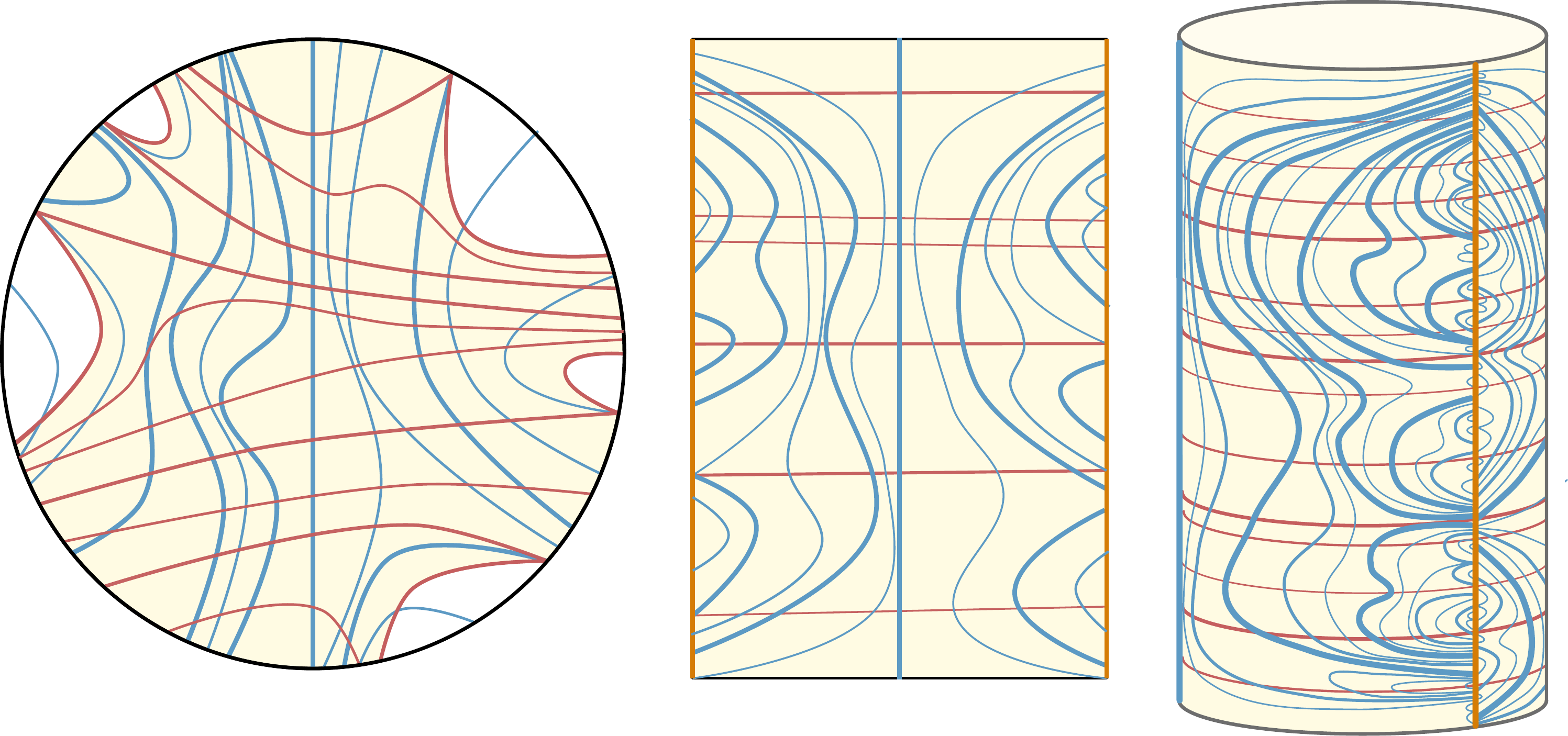}
	\caption{Constructing a cylinder $E_\infty\mid_{\ell^s}$ in the non $\R$-covered case. On the left is $\mc S^u(\ell^s)$, in the middle is the partial compactification, and on the right is $E_\infty\mid_{\ell^s}$. The marker $\Phi(\ell^s)$ appears as the left side of the cylinder $E_\infty\mid_{\ell^s}$.}\label{fig:nonR_ex}
\end{figure}

In the case that $\varphi$ is non $\R$-covered, we can reconstruct $E_\infty$ from $\orb$ by building it over lines in the leaf space and gluing the resulting cylinders together. This is easiest if we work with lines in $\mc L^u$ coming from leaves in $\mc L^s$. In this case, \Cref{fig:nonR_ex} illustrates what $\mc S^u(\ell^s)$, the partial compactification, and $E_\infty\mid_{\ell^s}$ might look like. One can glue together all cylinders $E_\infty\mid_{\ell^s}$ as $\ell^s$ ranges over all leaves of $\orb^s$ to build all of $E_\infty$.

\subsubsection{Twisting and the structure of $E_\infty$ over leaves of $\orb^s$}\label{sec:twisting}
Let $t$ be a transversal to $\orb^u$, let $m$ be a marker contained in a cylinder $E_\infty\mid_t$, and suppose $m$ limits onto the nonmarker section on both ends. Let $z_1$ be the lower limit point and $z_2$ be the upper endpoint. Then moving up from $z_1$, the marker $m$ might go left or right coming off the nonmarker section, and similarly $m$ might hit $z_2$ from the left or right. We will call a marker \defn{twisted} if it leaves and joins the nonmarker section from opposite sides, and \defn{untwisted} otherwise. It makes sense to describe an untwisted marker as lying on the left or right of the nonmarker section. See \Cref{fig:twisting} to see the various possible configurations. In the skew case, every marker is twisted.

	\begin{figure}[h!]
	\centering
	\includegraphics[width=0.6\linewidth]{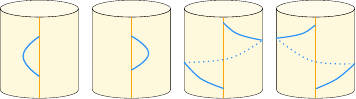}
	\caption{From left to right: an untwisted marker on the left of $s_\textbf{nm}$, an untwisted marker on the right of $s_\textbf{nm}$, and two twisted markers.}\label{fig:twisting}
\end{figure}

Given a leaf $\ell^s\in\mc L^s$ determining a line $\ell^s \subset \mc L^u$, the marker $\Phi(\ell^s)$ intersects every circle fiber in $E_\infty\mid_{\ell^s}$. This marker prevents other markers in $E_\infty\mid_{\ell^s}$ from twisting about the nonmarker seam, so every marker in this cylinder is untwisted. 

See \Cref{fig:twistingtubes} for an idea of what the markers on a larger piece of $E_\infty$ might look like.

\begin{figure}[h]
	\centering
	\includegraphics[width=0.7\linewidth]{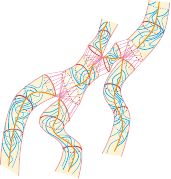}
	\caption{A possible configuration of markers on an open subset of $E_\infty$, including markers twisting in both directions.}\label{fig:twistingtubes}
\end{figure}

\subsection{The flow space perspective on sections of $E_\infty$}\label{sec:sections_in_orb}

Via \Cref{thm:stitching}, we can view the marker portion of sections of any universal circle as living in the orbit space. For $\partial\orb^+$ and $\partial\orb^-$, we describe these sets explicity below. This description will be used to prove \Cref{prop:distinct_flowucs}. First, we rephrase and expand upon \Cref{lem:gaps} using what we now know about the structure of markers in $E_\infty$.

\begin{lemma}\label{lem:marker_gap}
	Let $\lambda$ be a leaf of $\wt W^u$, so $\ell=\Theta(\lambda)$ is a frontier leaf of $\Omega^\pm_\lambda$. Let $c$ be the maximal chain in $\fr\Omega^\pm_\lambda$ containing $\ell$. Then
	\[
	\pi^\pm_\lambda(\spn (c))=\theta_\lambda=s_\textbf{nm}(\lambda),
	\] 
	and all other points in $\partial\orb$ are mapped to marker points of $\partial_\infty\lambda$ by $\pi^\pm_\lambda$.
	
	More specifically, for $\gamma$ a flowline in $\lambda$, let $x\in\partial_\infty\lambda$ be the forwards endpoint of $\gamma$, let $p=\Theta(\gamma)$, let $\ell^s$ be the leaf of $\orb^s$ through $p$, and let $z^\pm\in\partial\orb$ be its positive and negative limit points. Then 
	\[(\pi^\pm_\lambda)^{-1}(\Phi(p))=z^\pm\cup \spn(c'),\] where $c'$ is the chain in $\fr\Omega^\pm_\lambda$ made of leaves nonseparated from $\ell^s$, if such a chain exists. See \Cref{fig:marker_gap_lemma}. 
\end{lemma}

\begin{figure}[h]
	\centering
	\includegraphics[width=0.8\linewidth]{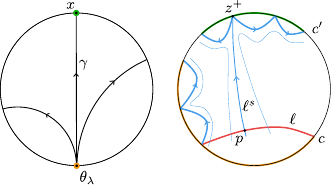}
	\caption{The setup for \Cref{lem:marker_gap}. On the left is the leaf $\lambda$, and on the right is $\orb$. In $\orb$, the chains $c$ and $c'$ are bolded, for the case of $\partial\orb^+$. The span of $c$ is in orange, and is sent to the orange point $\theta_\lambda$ in $\partial_\infty\lambda$ by $\pi^+_\lambda$. The span of $c'$ is in green, and is sent to the green point $x$ by $\pi^+_\lambda$.}\label{fig:marker_gap_lemma}
\end{figure}

\begin{proof}
	The first statement is a combination of \Cref{lem:gaps} and \Cref{lem:nonmarker}.
	
	For the second statement, by \Cref{thm:stitching}, we have \[\Phi(p)=x\in\partial_\infty\lambda.\] By \Cref{lem:gaps}, 
	\[\Theta\circ\wt h^\pm(\gamma)=\ell^s\cap\Omega^\pm_\lambda,\] and $\spn(c')$ (or just $z^\pm$ if $c'=\varnothing$) is exactly what is mapped to $x$ under $\pi^\pm_\lambda$.
\end{proof}

Using the above notation, we call $\spn(c)$ the \defn{nonmarker gap} of $\pi^\pm_\lambda$, and $\spn(c')$ a \defn{marker gap}.

\begin{figure}[h]
	\centering
	\includegraphics[width=0.8\linewidth]{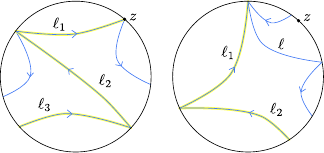}
	\caption{On the left, a backwards chain from $z$ is highlighted. The unhighlighted leaves are not part of such a chain. On the right, a backwards chain from $\ell$ opposite $z$ is highlighted. The unhighlighted leaves are not part of such a chain.}\label{fig:backwards_chain}
\end{figure}

Given a leaf $\ell$ of $\orb^s$, let $\ell^+$ and $\ell^-$ denote its positive and negative endpoints in $\partial\orb$, respectively. A \defn{chain in $\orb^s$ from a point $z\in\partial\orb$} is a sequence $\ell_1,\ell_2,\ldots,\ell_n$ in $\orb^s$ such that $z$ agrees with an endpoint of $\ell_1$, and $\ell_{i}$ and $\ell_{i+1}$ share an endpoint for all $i<n$. A chain from $z$ is a \defn{backwards chain} if $\ell_i^-=\ell_{i+1}^+$ for all $i<n$. Similarly, a chain is a \defn{forwards chain} if $\ell_i^+=\ell_{i+1}^-$ for all $i<n$. 

Given a leaf $\ell$ of $\orb^s$, a chain $\ell_1,\ell_2,\ldots,\ell_n$ in $\orb^s$ from an endpoint of $\ell$, and a point $z\neq \ell^\pm\in\partial\orb$, this chain is a \defn{chain from $\ell$ opposite $z$} if for all $i$, each endpoint of $\ell_i$ is in a distinct connected component of $\partial\orb\ssm\{\ell^+,\ell^-\}$ from $z$. See \Cref{fig:backwards_chain}.

Recall from \Cref{sec:tilting_uc} that each $z \in \partial \orb$ determines sections $\sigma_x^+$ and $\sigma_x^-$ of $\partial \orb^+$ and $\partial \orb^-$, respectively. To keep notation from getting out of hand, we often blur the distinction between a section and its image in $E_\infty$.

\begin{proposition}\label{prop:flow_section_in_orb}
	Let $z\in\partial\orb$, giving sections $\sigma_z^{+}$ of $\partial\orb^+$ and $\sigma_z^-$ of  $\partial\orb^-$ in $E_\infty$. Then 
	\begin{align*}\Phi^{-1}(\sigma^{+}_z\ssm s_\textbf{nm}) &= \{\ell_i\;\mid\;\ell_1,\ldots,\ell_n\text{ a backwards chain from }z\}\\
		&\cup \{\ell_i\;\mid\;\ell\in\orb^s,\;\ell_1,\ldots,\ell_n\text{ a backwards chain from }\ell\text{ opposite }z\}, 
		\end{align*} and 
	\begin{align*}\Phi^{-1}(\sigma^{-}_z\ssm s_\textbf{nm}) &= \{\ell_i\;\mid\;\ell_1,\ldots,\ell_n\text{ a forwards chain from }z\}\\
		&\cup \{\ell_i\;\mid\;\ell\in\orb^s,\;\ell_1,\ldots,\ell_n\text{ a forwards chain from }\ell\text{ opposite }z\}, 
	\end{align*}
\end{proposition}

\begin{proof}
We prove only the characterization of $\sigma_z^+$; the case for $\sigma_z^-$ is identical with the arrows on the leaves of $\orb^s$ reversed. 

Let 
\begin{align*}A &= \{\ell_i\;\mid\;\ell_1,\ldots,\ell_n\text{ a backwards chain from }z\},
\end{align*}
and let \[B=\{\ell_i\;\mid\;\ell\in\orb^s,\;\ell_1,\ldots,\ell_n\text{ a backwards chain from }\ell\text{ opposite }z\}.\]
We wish to show $\Phi(A\cup B)=\sigma^{+}_z\ssm s_\textbf{nm}$.

First suppose $p\in A$, so $p\in\ell_i$ for $\ell_1,\ldots,\ell_n$ a backwards chain from $z$. Let $\ell^u$ be the unstable leaf in $\orb^u$ through $p$, and let $\lambda^u$ be the corresponding leaf in $\wt W^u$. Consider the shadow $\Omega^+_{\lambda^u}$. By \Cref{lem:marker_gap}, we see that $\ell_i^+$ is a marker point for $\lambda^u$, i.e. $\pi_{\lambda^u}^{+}(\ell_i^+)\in \partial_\infty\l^u\subseteq E_\infty$ lies on a marker. 

\begin{figure}[h]
	\centering
	\includegraphics[width=0.8\linewidth]{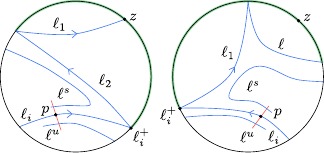}
	\caption{On the left, the case that $p\in A$. On the right, the case that $p\in B$. In both cases, using test leaves we see the green region of $\partial\orb$ gets collapsed to $\pi_{\lambda^u}^{+}(\ell_i^+)$.}\label{fig:test_leaf}
\end{figure}

In case $i=1$, we have $z=\ell_i^+$. Else, since $\ell_i$ and $\ell_{i-1}$ share an endpoint, $\ell_i^+$ must be in a gap for $\pi_{\lambda^u}^{+}$. To detect which other points in $\partial\orb$ lie in this gap, we use a leaf $\ell^s$ of $\orb^s$ through $\ell^u$ and near $p$ to approximate the chain of leaves spanning this gap. As in \Cref{fig:test_leaf}, we see that $z$ lies in this gap,
so 
\[
\pi_{\lambda^u}^{+}(z)=\pi_{\lambda^u}^{+}(\ell_i^+).
\] 

Thus, in either case we have
\[\Phi(p)=\pi_{\lambda^u}^{+}(\ell_i^+)=\pi_{\lambda^u}^{+}(z)=\sigma_z^{+}(\lambda^u),\]
so $\Phi(A)\subseteq \sigma^{+}_z\ssm s_\textbf{nm}$.

Now suppose $p\in B$, so $p\in\ell_i$ for $\ell_1,\ldots,\ell_n$ a backwards chain from $\ell$ opposite $z$ for some $\ell\in\orb^s$. A similar argument as above shows $\Phi(p)\in \sigma^{+}_z\ssm s_\textbf{nm}$; see \Cref{fig:test_leaf}. Thus,
\[\Phi(A\cup B)\subseteq \sigma^{+}_z\ssm s_\textbf{nm}.\]

\begin{figure}[h!]
	\centering
	\includegraphics[width=0.4\linewidth]{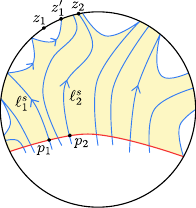}
	\caption{The cases when $\Phi(p)=\sigma_z^{+}(\lambda^u)$. When $p=p_1$, we have $z=z_1$ or $z=z_1'$. When $p=p_2$, we have $z=z_2$.}\label{fig:marker_gap}
\end{figure}

To show the opposite inclusion, let $p\in\orb$ such that $\Phi(p)=\sigma_z^{+}(\lambda^u)$ for some $\lambda^u\in\mc L^u$. Let $\ell^s$ be the stable leaf through $p$, and note that $\ell^u$ is the unstable leaf through $p$ (where by convention, $\ell^u=\Theta(\lambda^u)$). By \Cref{lem:marker_gap}, $z$ is either the forwards endpoint of $\ell^u$, or lies in a marker gap. In either case, we see a picture as in \Cref{fig:marker_gap}, and we can see the backwards chain realizing $p\in A\cup B$.
\end{proof}

Recall that a universal circle is \defn{admissible} if none of its sections cross markers. By \Cref{prop:flow_section_in_orb}, we have the following corollary.

\begin{corollary}\label{cor:admissible}
	The universal circles $\partial\orb^+$ and $\partial\orb^-$ are admissible.
\end{corollary}

\begin{remark}\label{rem:nm_is_flowspace_section}
 We also make the following observation about flow space sections in a cylinder $E_\infty\mid_{\ell^s}$, where $\ell^s$ is a leaf of $\orb^s$. Let $z^+\in\partial\orb$ denote the positive endpoint of $\ell^s$, and let $z^-$ denote the negative endpoint. Then by \Cref{prop:flow_section_in_orb}, the marker $\Phi(\ell^s)$ agrees with the flowpace sections $\sigma^{+}_{z^+}$ and $\sigma^-_{z^-}$ in this cylinder. Similarly, in this cylinder we have that $s_{\textbf{nm}}$, $\sigma^{+}_{z^-}$, and $\sigma^-_{z^+}$ agree.
\end{remark}

\subsubsection{Special sections in the flow space}
We can also map special sections to the flow space via \Cref{thm:stitching}. As we won't need this for any of the following arguments, we only sketch the case for $s^\ell_z$ with $z$ the nonmarker point of a leaf $\lambda$ of $\widetilde{\mc{F}}$. Special sections based on markers can be built out of similar pieces.

Let $\lambda\in\mc L^u$ and let $z=s_\textbf{nm}(\l)$. As before, we denote the stable saturation of $\ell=\Theta(\l)$ by $\mc S^s(\ell)$. Consider $\overline{\mc S^s(\ell)}$, the closure of $\mc S^s(\ell)$ in $\orb$. Then $\fr \overline{\mc S^s(\ell)}$ is a union of leaves of $\orb^s$, and these leaves are nonseparated from leaves in $\mc S^s(\ell)$. Using the orientations on $\orb^{s/u}$, these leaves each fall into one of four classes, denoted $X_{lu}$, $X_{ru}$, $X_{ld}$, and $X_{rd}$, as follows.

First, let $X_{lu} \cup X_{ru}$ be the set of stable leaves of $\fr \overline{\mc S^s(\ell)}$ that lie on the positive side of $\ell$ and let $X_{ld} \cup X_{rd}$ be the set that lie on negative side of $\ell$. Next, each leaf of $\fr \overline{\mc S^s(\ell)}$ is part of an oriented chain of nonseparated leaves starting with a unique stable leaf in $\mc S^s(\ell)$. If the branching of this chain is from below (along $\ell$) then all its leaves from $\fr \overline{\mc S^s(\ell)}$ are in $X_{ru} \cup X_{rd}$, and if the branching is from above then the leaves are in $X_{lu}\cup X_{ld}$. From this, we see that these sets partition the stable leaves of $\fr \overline{\mc S^s(\ell)}$. See also \Cref{fig:slz_in_orb}.

\begin{figure}[h]
	\centering
	\includegraphics[width=0.4\linewidth]{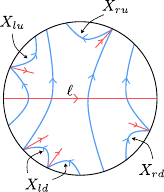}
	\caption{Leaves in each of $X_{lu}$, $X_{ru}$, $X_{ld}$, and $X_{rd}$, with respect to the leaf $\ell$, in red.}\label{fig:slz_in_orb}
\end{figure}

\begin{claim}
	If $\mc L^u_\lambda$ is the subspace of leaves comparable with $\lambda$, then
		\[
		\Phi^{-1}(s^\ell_z \ssm s_\textbf{nm}) \cap \mc L^u_\lambda  =X_{lu}\cup X_{rd}.
		\]
\end{claim}

Let $X=X_{lu}\cup X_{rd}$. First, one checks that that $\Phi(X)$ determines a section over $\mc L^u_\lambda$, i.e. it intersects any leaf $\ell^u$ comparable with $\ell$ at most once. To see that $\Phi(X)$ is in fact the special section based at $z$, first note the following. For leaves in $\ell^s\subseteq\mc L^u$ for $\ell^s$ a stable leaf through $\ell$, we see that the sets $X_{lu}$ and $X_{rd}$ are the ``upper left" and ``lower right" halves of the frontier of $\mc S^u(\ell^s)\cap\mc S^s(\ell)$ in $\mc S^u(\ell^s)$, as in the left of \Cref{fig:more_left}.
 
Now suppose that $\Phi(X)$ disagrees with $s^\ell_z$. Then at some leaf $\lambda^u$ in some cylinder, we see the section $s^\ell_z$ going more left up (or right down) than the section $\Phi(X)$. We can approximate the cylinder with a cylinder over a stable leaf. Mapping this back to $\orb$, we see that this is incompatible with the description of $X$ above. See the purple segment in \Cref{fig:more_left}; it necessarily crosses stable leaves (i.e. markers).

\begin{figure}[h]
	\centering
	\includegraphics[width=0.9\linewidth]{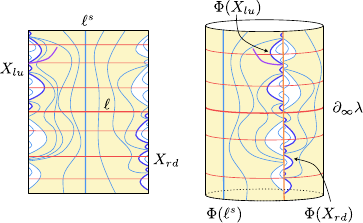}
	\caption{On the left: The rectangle is $\mc S^u(\ell^s)$ and the highlighted region is $\mc S^u(\ell^s)\cap\mc S^s(\ell)$. On the right: The image under $\Phi$. If $\Phi(X)$ were'nt leftmost, $s^\ell_z$ would contain something like the purple segment.}\label{fig:more_left}
\end{figure}

To extend this description to the rest of $\mc L^u$, we use the turning corners rule. 
Pick $\ell^u_1\in\ell^s\subseteq\mc L^u$ a leaf nonseparated from some $\ell^u_2$ comparable with $\ell$. By the turning corners rule, over lines through $\ell_2^u$, we have that \[s^\ell_z=s^\ell_x,\] where $x=s_\textbf{nm}(\ell^u_2)$. Thus, we can repeat the procedure stated above, and continue until we fill out all of $\mc L^u$. The same arguments as above show this builds $s^\ell_z$.

\section{The skew case}\label{sec:skew}
Before covering the non $\R$-covered case in \Cref{sec:distinct_ucs} and \Cref{sec:nonconjugate}, we briefly go over the skew case. While this material is already known to experts, we include it for completeness and to illustrate the use of \Cref{lem:leaf_shadow}, \Cref{thm:stitching}, and \Cref{prop:flow_section_in_orb} in a less complicated setting.

\subsection{Setup and notation}
Fix orientations on $\orb^{s/u}$ as in \Cref{fig:skew_setup}. Using \Cref{thm:stitching}, markers on $E_\infty$ are vertical, and the nonmarker curve is a helix as in \Cref{fig:skew_setup}.

\begin{figure}[h]
	\centering
	\includegraphics[width=\linewidth]{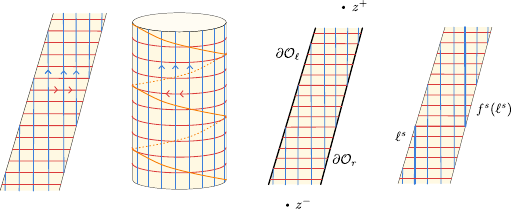}
	\caption{From left to right: (1) The orientations on $\orb^{s/u}$, (2) $E_\infty$, (3) The decomposition of $\partial\orb$ into $z^\pm$ and $\partial_{\ell/r}\orb$, (4) The one step up map $f^{s/u}$.}\label{fig:skew_setup}
\end{figure}

The boundary of the skew plane contains two points that aren't endpoints of leaves of $\orb^{s/u}$. These are the ``top" and ``bottom" of the infinite strip, and we'll denote them $z^+$ and $z^-$ respectively. The rest of the boundary is made of two open intervals $\partial_\ell \orb$ 
and $\partial_r\orb$, which each have a canonical identification with both $\mc L^u$ and $\mc L^s$ by considering the endpoints of leaves. We'll call these maps
\[h^{s/u}_{\ell/r}\colon \partial_{\ell/r}\orb \to\mc L^{s/u}.\]
See \Cref{fig:skew_setup}. 

We define the \defn{one step up} map
\[
f^{s/u}\colon \mc L^{s/u}\to\mc L^{s/u}
\] 
as in \Cref{fig:skew_setup}. See for example \cite[Section 2.3]{barthelme2025pseudo}.
Finally, we have a pair of involutions $\iota^{s/u}\colon\partial\orb\to\partial\orb$ 
defined by fixing $z^\pm$ and swapping endpoints of leaves of $\orb^s$ or $\orb^u$. 

\subsection{The Calegari--Dunfield universal circles}
A leftmost section $s^\ell_p$ in $E_\infty$ is either based on a marker or on the nonmarker section. In the first case, $s^\ell_p$ contains the marker it is based on, and follows the nonmarker curve elsewhere. In the second case, $s^\ell_p=s_\mathbf{nm}$ everywhere, as in \Cref{fig:skew_left}.

\begin{figure}[h]
	\centering
	\includegraphics[width=0.8\linewidth]{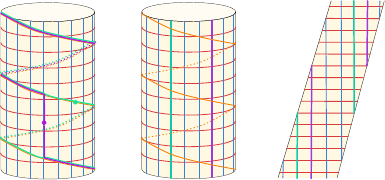}
	\caption{From left to right: (1) Leftmost sections in $E_\infty$, (2) Rightmost sections in $E_\infty$, (3) The images of rightmost sections in $\orb$ under $\Phi^{-1}$. }\label{fig:skew_left}
\end{figure}

The marker such a section is based at is parameterized by $\mc L^s$, and as the basepoint approaches either end of $\mc L^s$, the corresponding section converges pointwise to $s_\textbf{nm}$. The action of $\pi_1(M)$ on these sections is conjugate to the action on $\mc L^s$ when restricted to sections based on markers, and the nomarker section is a global fixed point. Thus, we can view $\mc C^\ell$ as the one point compactification of $\mc L^s$, with the corresponding action. This universal circle is minimal as a universal circle by construction, but isn't dynamically minimal due to the global fixed point.

A rightmost section is just a vertical line, as in \Cref{fig:skew_left}. Mapping to $\orb$ via $\Phi^{-1}$, we see that the markers making up a rightmost section are exactly an orbit of a leaf of $\orb^s$ under the one step up map $f$. Thus, we can identify $\mc C^r$ with its $\pi_1(M)$ action with $\mc L^s/f$. Via the map $h^s_r$, we can also view $\mc C^r$ as $\partial_r\orb/f$. Finally, the monotone maps for $\mc C^r$ are homeomorphisms that identify $\mc C^r$ to with $\partial_\infty \lambda$ for each $\lambda \in \mc L$.

This is again a minimal universal circle by construction. It is also dynamically minimal, since the action $\pi_1(M)\curvearrowright\mc L^s$ is.

\subsection{The flow space universal circles}
Now we tilt $\wt W^u$ to $\wt h_\pm(\wt W^u)$, defining universal circles $\partial\orb^\pm$ for $W^u$ as in \Cref{sec:tilting}. The action $\pi_1(M)\curvearrowright\partial\orb^\pm$ is just the action $\pi_1(M)\curvearrowright\partial\orb$. To find the monotone maps, it remains to describe the shadows.

Using \Cref{lem:leaf_shadow}, we see that the shadows $\Omega^\pm_\lambda$ are triangles, as in \Cref{fig:skew_shadows}. (This was originally described in \cite[Section 7]{fenley2005regulating}.) 
Thus, each monotone map $\pi^\pm_\lambda$ contains a single gap. For $\partial\orb^+$, the interval $\partial\orb_r$ is contained in every gap, and for $\partial\orb^-$, the interval $\partial\orb_\ell$ is contained in every gap. Thus, neither $\partial\orb^+$ nor $\partial\orb^-$ is minimal as a universal circle.

\begin{figure}
	\centering
	\includegraphics[width=0.5\linewidth]{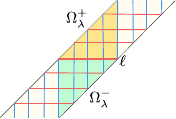}
	\caption{The shadows $\Omega^\pm_\lambda$ of a leaf $\lambda\in\mc L^u$.}\label{fig:skew_shadows}
\end{figure} 

Via \Cref{prop:flow_section_in_orb}, we see that the images of sections $\sigma^\pm_x$ in $\orb$ are single stable leaves, so the sets of sections of $\partial\orb^+$ and $\partial\orb^-$ are the same. The map $\iota^s\colon\partial\orb\to\partial\orb$
gives us an isomorphism of universal circles between $\partial\orb^+$ and $\partial\orb^-$.

Finally, collapsing $\partial_r \orb$ for $\partial\orb^+$ (or $\partial_\ell\orb$ for $\partial\orb^-$) gives us a monotone map to the circle $\mc C^\ell$, which was described above.

\section{Distinct universal circles}\label{sec:distinct_ucs}

In this section we prove \Cref{thm:distinct_ucs}. 

\begin{theorem}\label{thm:distinct_ucs}
	Suppose that $\varphi$ is not $\RR$-covered. Then the four universal circles $\mc C^\ell$, $\mc C^r$,  $\partial \orb^+$ and $\partial\orb^-$ of $W^u$ are distinct. Furthermore, no two of $\mc C^\ell$, $\mc C^r$, and $\partial\orb^+$ (respectively, $\partial\orb^-$) are covered by a common universal circle.
\end{theorem}

Part of \Cref{thm:distinct_ucs} is the content of  \Cref{prop:distinct_flowucs}.

\begin{proposition}\label{prop:distinct_flowucs}
	Suppose $\varphi$ is not $\R$-covered. Then $\partial\orb^+$ and $\partial \orb^-$ are nonisomorphic as universal circles.
\end{proposition}

Before proving \Cref{prop:distinct_flowucs}, we set the following notation. 
As in \Cref{sec:sections_in_orb}, given a leaf $\ell$ in $\orb^{s/u}$, let $\ell^+$ and $\ell^-$ denote its forward and backward endpoints in $\partial\orb$, respectively. A chain $\ell_1,\ldots,\ell_n$ of leaves in $\orb^s$ is said to be a \defn{maximal chain} if it is maximal with respect to inclusion among chains in $\orb^s$, i.e. it cannot be extended. 

A chain is \defn{oriented} if $\ell_i^+ = \ell_{i+1}^-$ for all $i<n$. An oriented chain is a \defn{maximal oriented chain} if it is maximal with respect to inclusion among all oriented chains, i.e. it cannot be extended as an oriented chain. Note that a maximal oriented chain is a maximal chain -- the existence of a leaf of $\orb^s$ extending the chain in a way that breaks orientability forces the existence of a leaf respecting the orientation of the chain. This follows from the fact that the stable (or unstable) leaves limiting to a fixed point in $\partial \orb$ have orientations that alternate.

As recorded in \Cref{sec:struct}, since $M$ is atoroidal, all chains in $\orb^s$ have finite length.

\begin{proof}[Proof of \Cref{prop:distinct_flowucs}]
	
	\begin{figure}[h]
		\centering
		\includegraphics[width=0.4\linewidth]{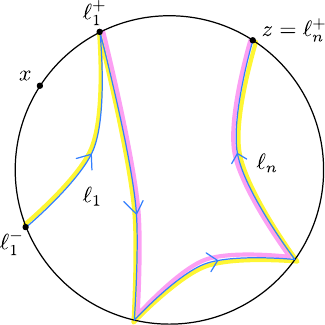}
		\caption{The leaves highlighted in yellow are part of $\Phi^{-1}(s^+_z)$, and the leaves highlighted in pink are part of $\Phi^{-1}(s^-_x)$.}\label{fig:distinct_flowspace_ucs}
	\end{figure}
	
	Let $\ell_1,\ldots,\ell_n$ be a maximal oriented chain in $\orb^s$ with $n\geq2$. Pick $x\in\partial\orb$ such that $x$ and $\ell_n^+$ lie in different connected components of $\partial\orb\ssm\{\ell_1^+,\ell_1^-\}$, as in \Cref{fig:distinct_flowspace_ucs}. Then by \Cref{prop:flow_section_in_orb}, $\Phi^{-1}(\sigma^-_x)$ contains $\ell_2$ through $\ell_n$, but not $\ell_1$.
	
	Suppose $\partial\orb^+\cong\partial\orb^-$. By \Cref{lem:same_sections}, we must have $\sigma^-_x=\sigma^+_z$ for some $z\in\partial\orb$. Since $\ell_n\subset \Phi^{-1}(\sigma^-_x)$, for $\ell_n$ to be in $\Phi^{-1}(\sigma^+_z)$ we must have that $\ell_n$ is in a backwards chain from $z$ or in a backwards chain from $\ell$ opposite $z$ for some leaf $\ell\in\orb^s$, as in \Cref{prop:flow_section_in_orb}. However, $\ell_n$ is the last leaf in a maximal chain, so the only possibility is that $z=\ell_n^+$. But then by \Cref{prop:distinct_flowucs} we see that $\ell_1\subset \Phi^{-1}(\sigma^+_z)$. Thus, no choice of $z$ works.
\end{proof}

We will now prove the remainder of \Cref{thm:distinct_ucs} by way of \Cref{lem:distinct_if_cross}. 

\begin{lemma}\label{lem:distinct_if_cross}
	
If two universal circles $\mc C_1$ and $\mc C_2$ have sections that cross, they are nonisomorphic and have no common cover.
\end{lemma}

\begin{proof}
	By \Cref{lem:same_sections}, isomorphic universal circles have the same set of sections. Since no two sections of a single universal circle cross, the circles can't be isomorphic.
	
	If $\mc C_1$ and $\mc C_2$ were covered by a common circle $\mc C$, the sections of $\mc C_1$ and $\mc C_2$ would all be sections of $\mc C$, so $\mc C$ would have sections that cross.
\end{proof}

Thus, we will analyse $\mathcal{C}^\ell$, $\mathcal{C}^r$, and $\partial\orb^+$ (respectively $\partial\orb^-$) by comparing their sets of sections in $E_\infty$. In light of \Cref{cor:admissible}, any differences of the sections of $\mathcal{C}^\ell$, $\mathcal{C}^r$, and $\partial\orb^\pm$ must occur around the nonmarker section. We will now prove \Cref{thm:distinct_ucs} by finding sections that cross at the nonmarker section. 

\begin{proof}[Proof of \Cref{thm:distinct_ucs}]
Let $t^s$ be a leaf of $\mc O^s$, determining a line in $\mc L^u$. Let $z\in\partial\orb$ be the backwards endpoint of $t^s$, and let $x$ be the forwards endpoint. By the description of shadows given in \Cref{lem:leaf_shadow}, $z$ lies in the nonmarker gap for every leaf $\lambda\in t ^s\subseteq\mc L^u$. Thus, in $E_\infty\mid_{t^s}$, the flow space sections $\sigma^{+}_z$ and $\sigma^-_x$ agree with the nonmarker section $s_\textbf{nm}$, as in \Cref{rem:nm_is_flowspace_section}.

\begin{figure}[h!]
	\centering
	\includegraphics[width=0.8\linewidth]{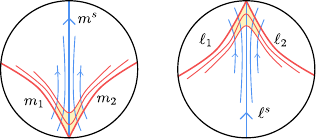}
	\caption{The configuration of the branching leaves in the proof of \Cref{thm:distinct_ucs}. Perfect fit rectangles are highlighted in yellow. The arrows on stable leaves indicate the corientation of $\orb^u$. }\label{fig:two_sided_branching}
\end{figure}

On the other hand, some section of $\mc C^\ell$ and of $\mc C^r$ must cross $s_\textbf{nm}$ in $E_\infty\mid_{t^s}$ by the following argument. Recall that the foliation $\widetilde{ W^u}$ has two-sided branching (\Cref{sec:struct}). This means there exists a pair of adjacent  nonseparated leaves $\lambda_1$ and $\lambda_2$ that are branching from below, and another pair $\mu_1$ and $\mu_2$ that are branching from above. Projecting to the flow space, we assume $\ell_1$ is on the left of $\ell_2$ and that $m_1$ is on the left of $m_2$. The $\ell_i$'s are separated by a stable leaf $\ell^s$, and the $m_i$'s are separated by a stable leaf $m^s$. See \Cref{fig:two_sided_branching} for the configurations of these leaves.

As $\varphi$ is a transitive Anosov flow, the $\pi_1(M)$-orbit of every leaf of $\widetilde{W^u}$ is dense in $\mc L^u$. Thus, we can move each of $\lambda_i$ and $\mu_i$ into $t^s\subseteq \mc L^u$. In particular, we do this so that moving from the bottom to the top of $t^s$, we see a translate of $g_1\cdot\lambda_1$ of $\lambda_1$, then $h_2\cdot\mu_2$, then $g_2\cdot\lambda_2$, and then $h_1\cdot\mu_1$, as in \Cref{fig:crossing_snm}. 

\begin{figure}[h!]
	\centering
	\includegraphics[width=0.9\linewidth]{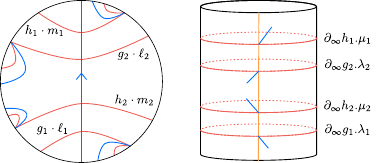}
	\caption{The configuration of translates of the branching leaves in $\orb$ and in  $E_\infty\mid_{t^s}$.}\label{fig:crossing_snm}
\end{figure}

By the description of $E_\infty$ over lines given in \Cref{subsec:over_lines}, translates of $\ell^s$ and $m^s$ give us markers in $E_\infty\mid_{t^s}$ under the map $\Phi$, arranged as in \Cref{fig:crossing_snm}. From here, it is easy to see that some sections of $\mc C^\ell$ and $\mc C^r$ must cross $s_\textbf{nm}\mid_{t^s}$. Consider the section of $\mc C^\ell$ based at a point on the marker $\Phi(g_1\cdot\ell^s)$. Going up, it will never take a marker coming off the right of $s_\textbf{nm}$, as that will never be further left than staying on $s_\textbf{nm}$. Thus, the section will cross $s_\textbf{nm}$ or travel along it. If it doesn't cross $s_\textbf{nm}$ before it gets to $\partial_\infty h_2\cdot\mu_2$, it certainly will at that point and follow along the marker $\Phi(h_1\cdot m^s)$. A similar argument for a rightmost section based at a point on $\Phi(g_2\cdot\ell^s)$ works for $\mc C^r$. Finally, it is clear from the picture that sections of $\mc C^\ell$ and $\mc C^r$ must cross each other.
\end{proof}

\begin{remark}\label{rem:totally_disconnected}
	In fact, the proof shows slightly more -- by the density of branching leaves, and thus the density of marker endpoints on $s_\textbf{nm}$, any special section (of $\mc C^\ell$ or $\mc C^r$) intersects $s_\textbf{nm}$ in a totally disconnected set.
\end{remark}

\section{Nonconjugate universal circle actions}\label{sec:nonconjugate}
Our goal in this section is to prove \Cref{thm:not_conj}. Note that the theorem holds when $\mc C^\ell$ is replaced by $\mc C^r$.

\begin{theorem}\label{thm:not_conj}
	Suppose $\varphi$ is not $\R$-covered. Then the actions $\pi_1(M)\curvearrowright\mathcal{C}^\ell$ and $\pi_1(M)\curvearrowright\partial\orb$ are not conjugate.
\end{theorem}

To see that the actions $\pi_1(M)\curvearrowright\mathcal{C}^\ell$ and $\pi_1(M)\curvearrowright\partial\orb$ aren't conjugate, we'll study the dynamics of a group element $g\in\pi_1(M)$ acting with a single fixed point in the interior of $\orb$. Such elements exist; in fact, the fixed points of these elements are dense by \cite[Lemma 2.30]{BFM}.
So, pick $g \in \pi_1(M)$ that fixes a unique $p \in \orb$ as well as the leaves $\ell^s$, $\ell^u$ of $\orb^s$ and $\orb^u$ with $p=\ell^s\cap\ell^u$.  Up to replacing $g$ by its inverse, we may suppose that it acts on $\ell^u$ with contracting dynamics and $\ell^s$ with expanding dynamics.

By \cite[Proposition 3.7]{BFM}, $g$ acts on $\partial\orb$ with four fixed points, the endpoints of $\ell^s$ and $\ell^u$, and these fixed points alternate between sinks and sources. We will see that $g$ acts on $\mathcal{C}^\ell$ with a single sink and a single source, so the two actions can't be conjugate. To see this, we need to understand the dynamics of $g$ on $E_\infty$.

Since $g$ fixes $\ell^s$, it fixes the cylinder $E_\infty\mid_{\ell^s}$. To understand the dynamics of $g$ on this cylinder, we use the stitching map $\Phi$. Keeping with our convention, we denote by $\lambda^u$ the leaf of $\wt W^u$ sent to $\ell^u$ by $\Theta$. We have that \[\Phi(\ell^u)=\partial_\infty\l^u\ssm s_\textbf{nm}(\l^u),\] and $\Phi(\ell^s)$ is a marker. Their intersection, $\Phi(p)$, is a fixed point for $g$. 
On $\Phi(\ell^u)$ we see contracting dynamics, and on $\Phi(\ell^s)$ we see expanding dynamics. The point $s_\textbf{nm}(\l^u)$ is also fixed, so $g$ acts on $\partial_\infty\l^u$ with sink source dynamics, where $s_\textbf{nm}(\l^u)$ is the sink and $\Phi(p)$ is the source.

\begin{figure}[h!]
	\centering
	\includegraphics[width=0.5\linewidth]{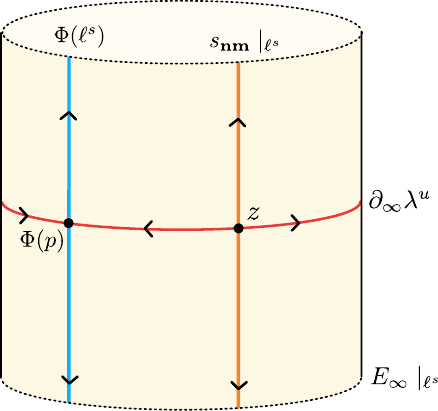}
	\caption{The dynamics of $g$ on $E_\infty\mid_{\ell^s}$.}\label{fig:non_conj_setup}
\end{figure} 

As $p$ is the only fixed point of $g$ in $\orb$, $\Phi(p)$ is the only fixed point in $E_\infty\ssm s_\textbf{nm}$. As for $s_\textbf{nm}$, the dynamics of $g$ are conjugated by the map $h\colon \mathcal{L}^u\rightarrow s_\textbf{nm}$ defined by $h(\lambda)=s_\textbf{nm}(\lambda)$. Thus, the only fixed point on $s_\textbf{nm}$ is $s_\textbf{nm}(\l^u)$, and $g$ acts on $s_\textbf{nm}\mid_{\ell^s}$ with expanding dynamics. To simplify notation, let $z=s_\textbf{nm}(\l^u)$. \Cref{fig:non_conj_setup} illustrates the situation so far.

It's clear there are at least two fixed points of $g$ on $\mathcal{C}^\ell$, those corresponding to the special sections $s^\ell_{\Phi(p)}$ and $s^\ell_z$. We will show that any other point in $\mathcal{C}^\ell$ converges to $s^\ell_z$ under backward iteration by $g$ and to $s^\ell_{\Phi(p)}$ under forward iteration, showing that $s^\ell_{\Phi(p)}$ and $s^\ell_z$ are the only two fixed points. We do this by showing that it is true for special sections based at points on $\partial_\infty\l^u$. Then by picking such a section in each connected component of $\mathcal{C}^\ell\ssm\{s^\ell_z,s^\ell_{\Phi(p)}\}$, we have that $g$ acts on $\mathcal{C}^\ell$ with sink source dynamics.

The first step towards showing this is noting that near $z$ there are arbitrarily small markers. This will imply that special sections based at points near $z$ quickly agree with $s^\ell_z$. We make the first claim precise below.

\begin{lemma}\label{lem:small_markers}
	Let $p\in\orb$, $g\in\pi_1(M)$, $\ell^u\in\orb^u$, and $\ell^s\in\orb^s$ be as in the setup so far. Then for any neighborhood \[U\subseteq\ell^s\subseteq\mathcal{L}^u\] of $\l^u$, there are maximal 
markers $m_U^\ell,m_U^r\subseteq E_\infty\mid_{\ell^s}$ that instersect $\partial_\infty\l^u$, are contained in $E_\infty\mid_U$, and such that $m_U^\ell$ is to the left of $s_\textbf{nm}\mid_{\ell^s}$ and $m_U^r$ is to the right, in the sense of \Cref{sec:twisting}.
\end{lemma}

\begin{figure}[h]
	\centering
	\includegraphics[width=\linewidth]{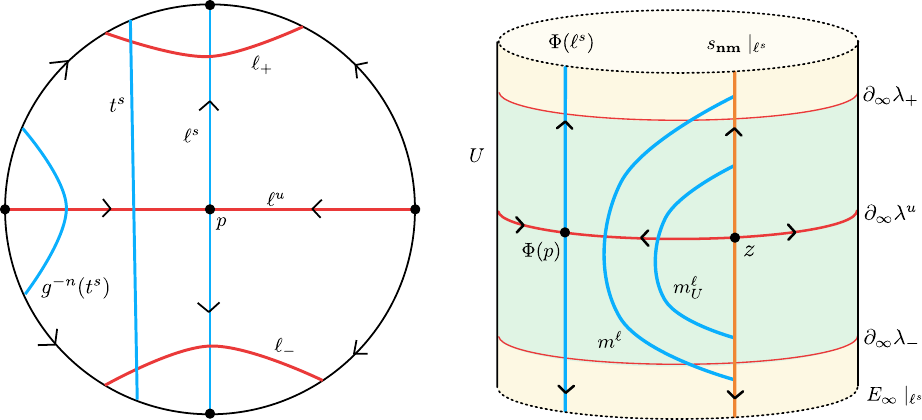}
	\caption{The setup in the proof of \Cref{lem:small_markers}. On the left is the setup in $\orb$, and on the right we translate this to $E_\infty\mid_{\ell^s}$. The region $U$ is highlighted in green.}
	\label{fig:small_markers}
\end{figure}

\begin{proof}
	 Let $U$ be as above, and let $\lambda_+$ and $\lambda_-$ be the leaves of $\mc L^u$ such that $\partial_\infty\lambda_+$ and $\partial_\infty\lambda_-$ bound $U$ from above and below respectively. See \Cref{fig:small_markers}.

	Let $m^\ell=\Phi(t^s)$ be any marker intersecting $\partial_\infty\l^u$ between $z$ and $\Phi(p)$ and to the left of $s_\textbf{nm}\mid_{\ell^s}$. We claim that \[g^{-n}(m^\ell)\subseteq E_\infty\mid_{\ell^s}\] for some $n\geq0$. To see this, we work in $\orb$ and use what we know about the dynamics of $g$ on $\partial\orb$. The endpoints of $t^s$ approach the righthand endpoint of $\ell^u$ under negative iteration.
		
	 Letting \[m_U^\ell=g^{-n}(m^\ell),\] we've found a marker that fits the bill.
	
	Starting with an analogous marker $m^r$ to the right of $s_\textbf{nm}\mid_{\ell^s}$ and following the same construction yields $m_U^r$.
\end{proof}

Now we are ready to show that $s^\ell_z$ is a source, at least for special sections based at points on $\partial_\infty\l^u$.

\begin{lemma}\label{lem:special_source}
	Let $x\in\partial_\infty\l^u\ssm\{\Phi(p)\}.$ Then \[g^{-n}(s^\ell_x)\rightarrow s_z^\ell\text{ as }n\rightarrow\infty.\]
\end{lemma}

\begin{proof}
	First we'll show that for all leaves $\nu$ incomparable with $\l^u$, we have \[s^\ell_{g^{-n}(x)}(\nu)=s^\ell_z(\nu)\] for all $n\geq 0$.
	
	This is immediate from the ``turning corners" rule for extending special sections. Let $\gamma$ be the minimal broken path in $\mathcal{L}^u$ from $\l^u$ to $\nu$. As $\l^u$ and $\nu$ are incomparable, $\gamma$ must pass through pairs of nonseparated leaves. Let $\mu_1$ and $\mu_2$ be the first such pair. Then while $s^\ell_{g^{-n}(x)}$ and $s^\ell_z$ may differ on $\mu_1$, the are both nonmarker on $\mu_2$, and from this point forward they agree as they travel along $\gamma$. See \Cref{fig:agree_after_turn}.
	
	\begin{figure}[h!]
		\centering
		\includegraphics[width=0.8\linewidth]{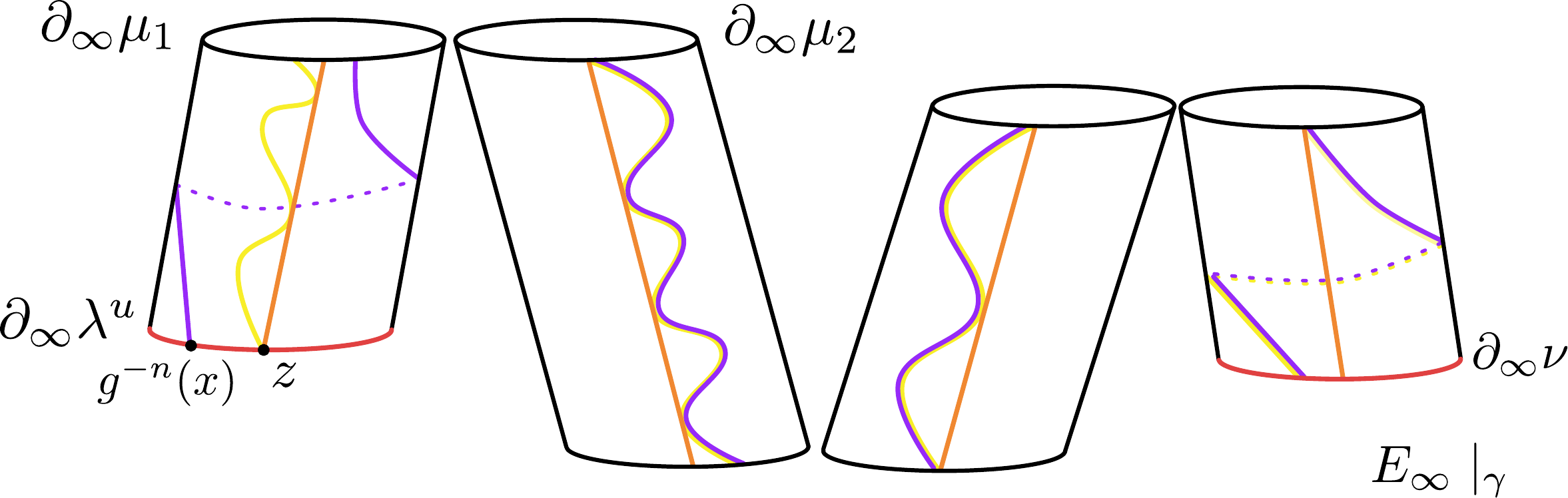}
		\caption{The sections $s^\ell_{g^{-n}(x)}$ (in purple) and $s^\ell_z$ (in yellow) agree after the first turn.}\label{fig:agree_after_turn}
	\end{figure}

	Now we'll show that for any neighborhood $U\subseteq\ell^s\subseteq\mathcal{L}^u$ of $\l^u$, there exists some $N\geq0$ such that for all $n\geq N$, \[s^\ell_{g^{-n}(x)}(\nu)=s^\ell_z(\nu)\] for all leaves $\nu$ comparable with $\l^u$ that are outside of $U$.
	
	To this end, first suppose that $x$ is to the left of $s_\textbf{nm}(\l^u)$, i.e. $(\Phi(p),x,s_\textbf{nm}(\l^u))$ is a counterclockwise oriented triple. Let $U$ be such a neighborhood. By \Cref{lem:small_markers}, there exists a marker $m_U^r$ contained in $E_\infty\mid_{U}$, intersecting $\partial_\infty\l^u$, and to the right of $s_\textbf{nm}$. As $m_U^r$ determines a subneighborhood $U'$ of $U$, we can apply \Cref{lem:small_markers} again to get a marker $m^\ell_{U'}$, this time on the left. \Cref{fig:inside_marker} illustrates the setup.
	
		\begin{figure}[h!]
		\centering
		\includegraphics[width=0.7\linewidth]{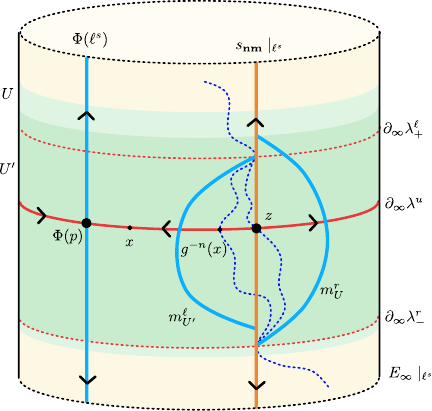}
		\caption{The setup for the proof of \Cref{lem:special_source}. The neighborhood $U$ is shaded in light green, and $U'$ is shaded in darker green. $N$ is chosen large enough so that $g^{-N}$ moves the point $x$ inside the crescent bound by $m^\ell_{U'}$ and $s_\textbf{nm}$. The markers $m_{U'}^\ell$ and $m^r_U$ ``fence in" special sections passing through the crescents.}\label{fig:inside_marker}
	\end{figure}
	
	By the dynamics of $g$ on $E_\infty\mid_{\ell^s}$, there exists a constant $N$ such that for $n\geq N$, $g^{-n}(x)$ is inside the crescent region bound by $m^\ell_{U'}$ and $s_\textbf{nm}$, as in \Cref{fig:inside_marker}. Then $m^\ell_{U'}$ and $m^r_U$ ``fence in" $s^\ell_{g^{-n}(x)}$ as follows.
	
	Let $\lambda^\ell_+$ be the leaf at the upper limit of $m^\ell_{U'}$, and let $\lambda^r_-$ be the leaf at the lower limit of $m^r_{U}$. As $s^\ell_{g^{-n}(x)}$ travels up from $g^{-n}(x)$, it can't cross the marker $m^\ell_{U'}$. Thus, it must agree with the nonmarker curve at $\lambda^\ell_+$. The same reasoning applies to $s^\ell_z$, so $s^\ell_{g^{-n}(x)}$ and $s^\ell_z$ agree on all leaves at and above $\lambda^\ell_+$.
	
	Similar reasoning shows that $s^\ell_{g^{-n}(x)}$ and $s^\ell_z$ agree on all leaves at and below $\lambda^r_-$. Thus, in the case that $x$ was to the left of the nonmarker curve, we've shown that $s^\ell_z$ and $s^\ell_{g^{-n}(x)}$ agree on all leaves comparable to $\l^u$ and outside of $U$.
	
	To handle the case that $x$ is to the right of the nonmarker curve, we essentially swap ``left" and ``right" in the above proof. First pick a marker $m^\ell_U$ on the left, and then pick a marker $m^r_{U'}$ on the right, where $U'$ is defined by $m^\ell_U$. Pick $N$ high enough so $g^{-N}$ moves $x$ into the crescent bound by $m^r_{U'}$ and the nonmarker curve, and repeat the fencing in argument.
	
	To finish the proof, let $\nu\in\mathcal{L}^u$ be any leaf. If $\nu$ and $\l^u$ are incomparable, we've shown that 
		\[s^\ell_{g^{-n}(x)}(\nu)=s^\ell_z(\nu)\] 
for all $n\geq 0$. If $\nu$ and $\ell^u$ are comparable, pick a small neighborhood $U\subseteq\ell^s\subseteq\mathcal{L}^u$ of $\l^u$ not containing $\nu$. Then we've shown that for high enough $n$,
	
	\[s^\ell_{g^{-n}(x)}(\nu)=s^\ell_z(\nu).\] 
	
	As we have pointwise eventual equality between $s^\ell_{g^{-n}(x)}$ and $s^\ell_z$, we have pointwise convergence, so 
	
	\[g^{-n}(s^\ell_x)\rightarrow s_z^\ell\text{ as }n\rightarrow\infty.\]
	\end{proof}

Next we show that $s^\ell_{\Phi(p)}$ is a sink, at least for special sections based on $\partial_\infty\l^u$. To do this, we need one more preliminary lemma. 

\begin{lemma}\label{lem:prop_emb}
	The line $\ell^s\subseteq\mathcal{L}^u$ is properly embedded.
\end{lemma}

\begin{proof}
	Suppose not, so there exists $$\tau^u\in\overline{\ell^s}\ssm\ell^s,$$ where the closure is taken in $\mathcal{L}^u$. Then $t^u=\Theta(\tau^u)$ is an unstable leaf not intersecting $\ell^s$, but that is a limit of unstable leaves $t_i$ that do intersect $\ell^s$. The accumulation set of the $t_i$ in $\ol \orb$ includes an endpoint of $\ell^s$ in $\partial \orb$ and so 
 $\ell^s$ makes a perfect fit with $t^u$ or a leaf nonseparated from $t^u$. However, $g$ would then also fix a point in $t^u$ and 
this contradicts the fact that $g$ acts on $\orb$ with exactly one fixed point.
\end{proof}

We are now ready to prove the following. The proof is similar to but slightly easier than the proof of \Cref{lem:special_source}.

\begin{lemma}\label{lem:special_sink}
		Let $x\in\partial_\infty\l^u\ssm\{z\}$. Then \[g^{n}(s^\ell_x)\rightarrow s_{\Phi(p)}^\ell\text{ as }n\rightarrow\infty.\]
\end{lemma}

\begin{proof}
	Let $x$ as in the statement of the lemma. First we'll show that for $\nu\in\ell^s\subseteq\mathcal{L}^u$, \[s^\ell_{g^n(x)}(\nu)\rightarrow s^\ell_{\Phi(p)}(\nu)\text{ as }n\rightarrow\infty.\]
	By the dynamics of $g$ on $E_\infty\mid_{\ell^s}$, for high enough $n$, $g^n(x)$ will lie on a marker $m$ that extends past $\partial_\infty\nu$. See \Cref{fig:special_sink}. The rectangular region $R$ bounded by $\partial_\infty\l^u$, $\partial_\infty\nu$, $m$, and $\Phi(\ell^s)$ is trivially foliated by markers, so we have pointwise convergence \[s^\ell_{g^n(x)}(\tau)\rightarrow \Phi(\ell^s)\cap \tau\] for all unstable leaves $\tau$ in this region, and in particular for $\nu$. As \[s^\ell_{\Phi(p)}\mid_{\ell^s}=\Phi(\ell^s)\mid_{\ell^s},\] we have handled this case.
	
	\begin{figure}[h!]
		\centering
		\includegraphics[width=0.5\linewidth]{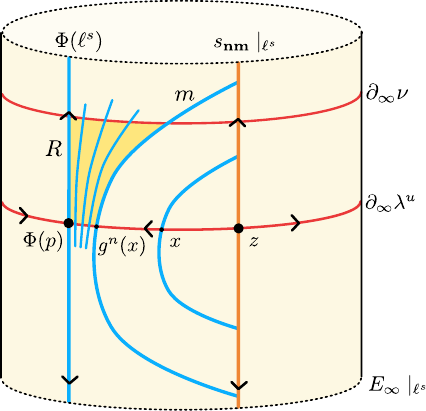}
		\caption{Under the action of $g$ on $E_\infty\mid_{\ell^s}$, markers converge to $\Phi(\ell^s)$. The rectangular region $R$ is highlighted in yellow.}\label{fig:special_sink}
	\end{figure}
	
	Now suppose $\nu\in\mathcal{L}^u\ssm\ell^s$. By \Cref{lem:prop_emb}, this means the minimal path $\gamma$ from $\l^u$ to $\nu$ passes through a leaf $\mu$ nonseparated from a leaf in $\ell^s\subseteq\mc L^u$. As in the proof of \Cref{lem:special_source}, the sections $s^\ell_{g^n(x)}$ and $s^\ell_{\Phi(p)}$ agree after the first turn, so for all such $\nu$ and for all $n$ we have \[s^\ell_{g^n(x)}(\nu)=s^\ell_{\Phi(p)}(\nu).\]
	
	Thus, as in the proof of \Cref{lem:special_source}, we have shown pointwise convergence at all $\nu\in\mathcal{L}^u$, which completes the proof.
\end{proof}

Putting \Cref{lem:special_source} and \Cref{lem:special_sink} together, we can show the following.

\begin{proposition}\label{prop:nconj_sink_source}
	For any section $s\in\mathcal{C}^\ell\ssm \{s^\ell_{\Phi(p)},s^\ell_z\}$, we have \[g^{-n}(s)\rightarrow s^\ell_z\text{ as }n\rightarrow\infty\] and 
	\[g^{n}(s)\rightarrow s^\ell_{\Phi(p)}\text{ as }n\rightarrow\infty.\]
	In particular, $g$ acts on $\mathcal{C}^\ell$ with exactly two fixed points, a sink and a source.
\end{proposition}

\begin{proof}
	We follow the proof strategy described earlier. The space \[X=\mathcal{C}^\ell\ssm\{s^\ell_z,s^\ell_{\Phi(p)}\}\] has two components. Choosing $x_1$ and $x_2$ in opposite components of \[\partial_\infty\l^u\ssm\{z,\Phi(p)\},\] we have that $s^\ell_{x_1}$ and $s^\ell_{x_2}$ lie in different components of $X$. By \Cref{lem:special_source} and \Cref{lem:special_source}, we have that 
	\[g^{-n}(s^\ell_{x_i})\rightarrow s^\ell_z\text{ as }n\rightarrow\infty\] and 
	\[g^{n}(s^\ell_{x_i})\rightarrow s^\ell_{\Phi(p)}\text{ as }n\rightarrow\infty\] for $i=1,2$. Thus, any section $s\in\mathcal{C}^\ell$ is sandwiched between $g$-orbits of either $s^\ell_{x_1}$ or $s^\ell_{x_2}$, so we see that $s^\ell_z$ is a global source and $s^\ell_{\Phi(p)}$ is a global sink for the action of $g$ on $\mathcal{C}^\ell$.
\end{proof}

As $g$ acts on $\partial\orb$ with four fixed points, this proves \Cref{thm:not_conj}.

\bibliographystyle{alpha}
\bibliography{universal_circles.bbl}

\end{document}